\documentclass[draft]{amsart}
\usepackage{cases}
\usepackage{amsfonts}
\usepackage{amscd}
\usepackage{amssymb}

\usepackage{microtype}

\allowdisplaybreaks

\usepackage{amsmath, amsthm, amscd, amsfonts, amssymb, graphicx, color}
\usepackage[active]{srcltx}

\usepackage[UKenglish]{babel}
\usepackage[UKenglish]{datetime}

\newcommand{\RR}{\mathbb{R}}

\newcommand{\CC}{\mathbb{C}}
\newcommand{\FF}{\mathbb{F}}

\newcommand{\ulb}{{\textup{(}}}
\newcommand{\urb}{{\textup{)}}}

\newcommand{\ba}{A}
\newcommand{\na}{A}
\newcommand{\uniba}{B}
\newcommand{\bs}{X}
\newcommand{\ns}{X}

\newcommand{\rba}{{\pi}}

\newcommand{\leftreg}{{\ell}}
\newcommand{\leftcent}[1]{{\mathcal M}_{\leftreg}(#1)}

\newcommand{\norm}[1]{{\| #1\|}}

\newcommand{\abs}[1]{\left\vert #1\right\vert}

\newcommand{\unit}{1_{\uniba}}

\newcommand{\e}{e_\lambda}
\newcommand{\bl}{b(\lambda)}

\newcommand{\net}[1]{\{{#1}_\lambda\}_{\lambda\in \Lambda}}
\newcommand{\seqk}[1]{\{{#1}_k\}_{k=1}^{\infty}}
\newcommand{\seqko}[1]{\{{#1}_k\}_{k=0}^{\infty}}
\newcommand{\ai}{\net{e}}
\newcommand{\fai}{\{f(\e)\}_{\lambda\in \Lambda}}
\newcommand{\fe}{f(\e)}

\newcommand{\largestlambda}{\lambda_{\textup{la}}}
\newcommand{\elargest}{e_{\largestlambda}}

\newcommand{\bounded}{\mathcal B}

\newcommand{\inv}[1]{\mathrm{Inv}(#1)}

\newcommand{\pos}[1]{#1^+}

\newcommand{\Ell}{\mathrm{L}}

\newcommand{\obs}{X}
\newcommand{\oba}{A}
\newcommand{\unioba}{B}

\newcommand{\ess}[1]{#1_\mathrm{e}}

\newcommand{\onefunction}{\mathbf 1}

\newcommand{\ts}{\Omega}
\newcommand{\cont}[1]{\mathrm{C}(#1)}
\newcommand{\conto}[1]{\mathrm{C}_0(#1)}
\newcommand{\contc}[1]{\mathrm{C}_{\mathrm c}(#1)}
\newcommand{\contb}[1]{\mathrm{C}_{\mathrm b}(#1)}

\newcommand{\boundedop}[1]{\bounded(#1)}

\theoremstyle{plain}

\newtheorem{theorem}{Theorem}[section]
\newtheorem{proposition}[theorem]{Proposition}
\newtheorem{lemma}[theorem]{Lemma}
\newtheorem{corollary}[theorem]{Corollary}

\newtheorem*{lemma*}{}

\theoremstyle{definition}

\newtheorem{example}[theorem]{Example}
\newtheorem{remark}[theorem]{Remark}

\numberwithin{equation}{section}

\begin{document}

\bibliographystyle{plain}


\title [Simultaneous power factorization]{Simultaneous power factorization in modules over Banach algebras}

\author{Marcel de Jeu}
\address{Marcel de Jeu, Mathematical Institute, Leiden University, P.O.\ Box 9512, 2300 RA Leiden, The Netherlands}
\email{mdejeu@math.leidenuniv.nl}

\author{Xingni Jiang}
\address{Xingni Jiang, Mathematical Institute, Leiden University, P.O.\ Box 9512, 2300 RA Leiden, The Netherlands}
\email{xingnijiang@gmail.com}




\begin{abstract}
Let $A$ be a Banach algebra with a bounded left approximate identity $\{e_\lambda\}_{\lambda\in\Lambda}$, let $\pi$ be a continuous representation of $A$ on a Banach space $X$, and let $S$ be a non-empty subset of $X$ such that $\lim_{\lambda}\pi(e_\lambda)s=s$ uniformly on $S$. If $S$ is bounded, or if $\{e_\lambda\}_{\lambda\in\Lambda}$ is  commutative, then we show that there exist $a\in A$ and maps $x_n: S\to X$ for $n\geq 1$ such that $s=\pi(a^n)x_n(s)$ for all $n\geq 1$ and $s\in S$. The properties of $a\in A$ and the maps $x_n$, as produced by the constructive proof,  are studied in some detail. The results generalize previous simultaneous factorization theorems as well as Allan and Sinclair's power factorization theorem.
In an ordered context, we also consider the existence of a positive factorization for a subset of the positive cone of an ordered Banach space that is a positive module over an ordered Banach algebra with a positive bounded left approximate identity. Such  factorizations are not always possible. In certain cases, including those for positive modules over ordered Banach algebras of bounded functions, such positive factorizations exist, but the general picture is still unclear. Furthermore, simultaneous pointwise power factorizations for sets of bounded maps with values in a Banach module (such as sets of bounded convergent nets) are obtained. A worked example for the left regular representation of $\mathrm{C}_0({\mathbb R})$ and unbounded $S$ is included.
\end{abstract}

\maketitle


\section{Introduction and overview}\label{introduction}

Let $\ba$ be a real or complex Banach algebra with a bounded left approximate identity $\ai$, and let $\rba$ be a continuous representation of $\ba$ on a Banach space $\bs$. Let $s \in \bs$, and let $\epsilon>0$. Using a notation that anticipates what is to come, the Cohen factorization theorem (see \cite{cohen1959} for the original case where $\bs=\ba$) states that, if $\lim_\lambda \rba(\e)s=s$, then there exist $a\in \ba$ and $x_1(s)\in\bs$ such that $s=\rba(a)x_1(s)$  and $\norm{s-x_1(s)}<\epsilon$.  This result has been generalized in two different directions. First of all, there are results that are concerned with simultaneously factoring all elements $s$ of a subset $S$ of $\bs$ that need not be a singleton. In that case, one wants to establish the existence of $a\in\ba$ and a map $x_1: S\to\bs$ such that  $s=\rba(a)x_1(s)$ for all $s\in S$, together with some additional properties of $a$ and  $x_1$. This is possible, for example, if $S$ is bounded and such that $\lim_{\lambda}\rba(\e)s=s$ uniformly on $S$; see e.g.\ \cite[Theorem~17.1]{doran}. According to \cite[p.~251]{doran}, this line of research started with Ovaert (see~\cite{Ovaert}) and Craw (see~\cite{Craw}). In \cite[Section~17]{doran}, such results are called `multiple factorizations', but we prefer `simultaneous factorizations'.  The term `multiple factorizations' is a bit ambiguous, since it could also be felt to be related to the second type of generalization of Cohen's original result, namely, the power factorization as established by Allan and Sinclair in \cite{Allan1976}. Here $S$ consists of one element $s$ such that $\lim_{\lambda}\rba(\e)s=s$ again, but this time the existence is established of $a\in\ba$ and, for all $n\geq 1$, an element $x_n(s)$ of $\bs$ such that $s=\rba(a^n)x_n(s)$, together with some additional (and now more sophisticated) properties of $a$ and the elements $x_n(s)$ for $n\geq 1$.

We refer to  \cite{Dalesbook}, \cite{doran}, \cite{Kisynski}, and \cite{Palmer_I} for a further description of the historical development concerning factorizations and additional references,  including those for actions of Fr\'echet algebras on Fr\'echet spaces.

In this paper, we combine these two types of generalizations. If $S\subset\bs$ is such that $\lim_{\lambda}\rba(\e)s=s$ uniformly on $S$, and if $S$ is bounded or $\ai$ is commutative, then, according to our main result Theorem~\ref{thm:simultaneous_power_factorization},  there exist $a\in \ba$ and maps $x_n:S\to\bs$ for $n\geq 1$ such that $s=\rba(a^n)x_n(s)$ for all $s\in S$ and $n\geq 1$.  We shall refer to this as a simultaneous power factorization. We are not aware of previous simultaneous factorization results where $S$ need not be bounded. It should be noted here, however, that it was already observed for pointwise power factorization in \cite[p.~32-33]{Allan1976} that the situation where $\ai$ is commutative is more manageable than the general case.

Theorem~4.4 also contains several additional properties of $a$ and the maps $x_n$ for $n\geq 1$. Special cases of some of these properties coincide with facts that are already known for the simultaneous non-power and the pointwise power factorizations as described above.

The proof of Theorem~\ref{thm:simultaneous_power_factorization} is an adaptation of the proof of \cite{Allan1976} from the case of one point to that of a set. This makes the proof, which in \cite{Allan1976} is already a considerably more sophisticated version of Cohen's original technically ingenious proof, perhaps still a little bit more demanding to verify. We have, therefore, tried to be rather precise in our argumentation.

We have also tried to distil some non-obvious information concerning the form of the resulting factorization. These facts are by-products of the proof, which is constructive. In view of the size of the current proof, and the resulting degree of (in)accessibility thereof, it is hardly attractive any more for a reader to inspect the proof, and decide whether a certain statement is implicitly also proven. We thought, therefore, that it would be better to do a precise bookkeeping along the way, and, in the end, include all these by-products in this one main result Theorem~\ref{thm:simultaneous_power_factorization}.

This paper is organized as follows.

Section~\ref{sec:conventions_and_notation} contains a few conventions and some notation.

Section~\ref{sec:uniform_convergence_on_subsets_of_Banach_modules} is concerned with the condition that $\lim_{\lambda}\rba(\e)s=s$ uniformly on $S$. Later on, we shall need to know how this behaves under various operations on $S$. We also investigate how this property depends on the choice of $\ai$, and we introduce a basic example, namely, the left regular representation of $\conto{\RR}$, that will be the subject of Section~\ref{sec:worked_example}. Results such as Proposition~\ref{prop:specific_choice} will look familiar to the reader who has seen earlier proofs of factorization theorems. It seemed inevitable  to give the proof, since we need it in greater generality than is available in the literature, and since we also need to do our bookkeeping. In compensation, some preparatory results (Lemma~\ref{lem:uniform_convergence_is_preserved_under_the_analytic_function} and Corollary~\ref{cor:both properties_are_preserved_under_the_analytic_function}) are established in a greater generality that could have applications in other contexts.

Section~\ref{sec:simultaneous power factorization} is the heart of the paper. After two preparatory results under different hypotheses but with identical conclusions (Proposition~\ref{prop:simultaneous_power_factorization_for_bounded sets} for bounded $S$ and Proposition~\ref{prop:simultaneous_power_factorization_for_commutative_approximate identity} for commutative $\ai$), the main result Theorem~\ref{thm:simultaneous_power_factorization} on simultaneous power factorization can be established. Spread over several remarks, this section also contains a discussion of the result, including its relation to the literature. One point should be noted here, and that is the absence of the unitization of $\ba$. This seems to be ubiquitous in the existing proofs of factorizations, but there is no actual need for this. Its only role in the existing proofs is to be available as a unital superalgebra $\uniba$ of $\ba$ with the property that a given continuous representation of $\ba$ extends to a continuous unital representation of $\uniba$. Any other unital superalgebra with this property will do equally well. This means that the choice of this superalgebra can be adapted to the situation at hand, and it is for this reason that this freedom of choice of $\uniba$ has been incorporated into the structure. In Section~\ref{sec:positive_simultaneous_power_factorization} we shall actually benefit from this; see e.g.\ Remark~\ref{rem:centralizer_algebras} and, in particular, Theorem~\ref{thm:positive_simultaneous_power_factorization_for_banach_algebras_of_functions}.

Section~\ref{sec:positive_simultaneous_power_factorization} is concerned with positive factorizations. It contains a further refinement (Theorem~\ref{thm:positive_simultaneous_power_factorization}) of Theorem~\ref{thm:simultaneous_power_factorization}; see Remark~\ref{rem:ordered_theorem_implies_original_theorem} for an explanation why it is a refinement and not a special case.  The questions to be considered are quite natural. For example, restricting ourselves to positive pointwise non-power factorization: if $\oba$ is an ordered Banach algebra with a positive bounded left approximate identity $\ai$, if $\rba$ is a positive representation of $\oba$ on an ordered Banach space $\obs$, and if $s$ is an element of the positive cone $\pos{\obs}$ of $\obs$ such that $\lim_{\lambda}\rba(\e)s=s$, do there exist $a\in\pos{A}$ and $x_1(s)\in\pos{A}$ such that $s=\rba(a)x_1(s)$? As we shall see, such $a\in\pos{\oba}$ exists whenever the positive cone $\pos{\oba}$ of $\oba$ is closed, but it may be impossible (the example is due to Rudin; see \cite{rudin1957}) to  arrange that $x_1(s)\in\pos{\obs}$. There is presently no clear-cut answer available that distinguishes between possibility and impossibility, but it seems as if the built-in freedom concerning the superalgebra $\uniba$ could be of some use here. A positive simultaneous power factorization result for ordered Banach algebras of bounded functions, Theorem~\ref{thm:positive_simultaneous_power_factorization_for_banach_algebras_of_functions}, can be established precisely because of this freedom.

Section~\ref{sec:simultanous_power_factorization_for_sets_of_maps} combines the main result Theorem~\ref{thm:simultaneous_power_factorization} with an idea on the use of auxiliary Banach modules that, according to \cite[p.~251]{doran}, goes back to Collins and Summer (see \cite{CollinsSummers}) and Rieffel (see \cite{Rieffel1969}). The most general set-up seems to be as in Theorem~\ref{thm:simultaneous_power_factorization_for_maps}, which is a simultaneous pointwise power factorization for sets of maps with values in a Banach module. This can then be specialized to e.g.\ sets of bounded continuous functions or sets of bounded convergent nets.

In Section~\ref{sec:worked_example}, we return to the example of the left regular representation of $\conto{\RR}$. For a concretely given unbounded $S$, we know from Theorem~\ref{thm:positive_simultaneous_power_factorization_for_banach_algebras_of_functions} that a positive simultaneous power factorization is possible. Since the characteristic features of such an actually possible factorization have been included in Theorem~\ref{thm:simultaneous_power_factorization}, it should be doable to find a concrete factorization, starting from the Ansatz as provided by Theorem~\ref{thm:simultaneous_power_factorization}. With some perseverance, this can indeed be carried out.

\section{Conventions and notation}\label{sec:conventions_and_notation}

In this short section, we establish our conventions and notation.

Our algebras and vector spaces can be over the real or complex numbers, unless otherwise stated.

The set of invertible elements of a unital algebra $\ba$ is denoted by $\inv{\ba}$.
A net $\ai$ in an associative algebra $\ba$ is said to be commutative if $\e e_{\lambda^\prime}=e_{\lambda^\prime}\e$ for all $\lambda,\lambda^\prime\in\Lambda$. The unit element of a unital normed algebra is assumed to have norm 1.
A representation of a unital algebra on a vector space is not required to be unital.

An ordered vector space $\bs$ is a vector space that is ordered (also in the complex case) by a positive convex cone $\pos{\bs}\subset\bs$. The cone $\pos{\bs}$ need not be proper; in particular, it can be the whole space. The cone need not be generating. If the vector space is normed, it need not be closed.

An ordered algebra $\ba$ is an algebra that is ordered (also in the complex case) by a positive convex algebra cone $\pos{\ba}$. Again, $\pos{\ba}$ need not be proper, or be generating, or\textemdash if applicable\textemdash be closed. The unit element of a unital ordered algebra is assumed to be positive.

A positive representation of an ordered algebra $\oba$ on an ordered vector space $\obs$ is a representation $\rba$ of $\oba$ on $\obs$ such that $\rba(a)x\in\pos{\obs}$ for all $a\in\pos{\ba}$ and $x\in\pos{\bs}$. One could perhaps say that $\obs$ is then a positive $\oba$-module; there is no fixed terminology yet.

We shall be working with norm topologies only. A bounded subset of a normed space is a norm bounded subset, and a continuous map between normed spaces is continuous for the norm topologies. If $\ns$ is a normed space, then $\boundedop{\ns}$ denotes the bounded operators on $\ns$. A representation $\rba$ of an algebra $\ba$ on a normed space $\bs$ maps $\ba$ into $\boundedop{\bs}$. In line with our conventions, it is said to be continuous if it is a continuous map between the normed spaces $\ba$ and $\boundedop{\ns}$.

If $\rba:\ba\to\boundedop{\ns}$ is a continuous representation of a normed algebra $\ba$ on a normed space $\ns$, then the essential subspace $\ess{\ns}$ of $\ns$ is defined as
\[
\ess{\ns}=\overline{\textup{Span}\{\,\rba(a)x: a\in \ba,\,x\in\ns\,\}}.
\]
Then $\ess{\bs}$ is clearly invariant under $\rba(A)$, so that $\ess{\bs}$ affords a continuous representation of $\ba$ on $\ess{\bs}$. The representation is said to be non-degenerate if $\ess{\bs}=\bs$. If $\ba$ has a bounded left approximate identity $\ai$, then it is easily seen that
\[
\ess{\ns}=\{\,x\in \ns : \lim_{\lambda}\rba(\e)x=x\,\}.
\]

Although we shall occasionally speak of a Banach module over a Banach algebra, we shall usually speak of a representation of a Banach algebra on a Banach space, and\textemdash with the left regular representation and Section~\ref{sec:worked_example} as only possible exceptions\textemdash we shall also include the corresponding symbol in the notation. Since the norm of the representation repeatedly appears in the estimates, this seems to be a natural choice so as to avoid keep introducing the constant in the bilinear pairing between the algebra and the space time and time again.

If $\ts$ is a topological space, then $\contc{\ts}$, $\conto{\ts}$, and $\contb{\ts}$ denote the continuous functions on $\ts$ that have compact support, that vanish at infinity, and that are bounded, respectively.

\section{Uniform convergence on subsets of Banach modules}\label{sec:uniform_convergence_on_subsets_of_Banach_modules}

As a preparation for the main development, and as general background, this section contains a number of results revolving around the condition that $\lim_{\lambda}\rba(\e)s=s$ uniformly on $S$. Here, as elsewhere in this paper, $\rba$ is a continuous representation of a normed algebra $\ba$ on a normed space $\bs$ with non-empty subset $S$, and $\ai$ is an approximate left identity of $\ba$. We also introduce a basic example in the context of the left regular representation of $\conto{\RR}$ that will be taken up in detail in Section~\ref{sec:worked_example} again.

The first result to be mentioned is standard, and can e.g.\ be found in the literature as \cite[Lemma~17.2]{doran}. We include the proof for the convenience of the reader.

\begin{lemma}\label{lem:totally_bounded_subsets}
Let $\ba$ be a normed algebra that has a bounded left approximate identity $\ai$, let $\rba$ be a continuous representation of $\rba$ on a normed space $\ns$, and let $S$ be a totally bounded subset of $\ess{\ns}$. Then $\lim_{\lambda}\rba(\e)s=s$ uniformly on $S$.
\end{lemma}

\begin{proof}
If $\rba=0$, then $\ess{\ns}=\{\,0\,\}$, in which case the result is trivial.  So let us assume that $\rba\neq 0$. Let $M\geq 1$ be a bound for $\ai$. Let $\epsilon>0$ be given. Since $S$ is totally bounded, there exist $x_1,\ldots,x_n\in\ess{\ns}$ such that $S\subset\bigcup_{i=1}^n\{\,\,x\in\ns : \norm{x-x_i}<\min\left(\epsilon/(3\norm{\rba}M),\epsilon/3\right)\,\}$. Choose $\lambda^\prime\in\Lambda$ such that $\norm{\rba(\e)x_i-x_i}<\epsilon/3$ for all $i=1,\ldots,n$ and $\lambda\geq\lambda^\prime$. If $s\in S$, there exists $i_0$ such that $\norm{s-x_{i_0}}<\min(\epsilon/\left(3\norm{\rba}M\right),\epsilon/3)$. We now see that, for all $\lambda\geq\lambda^\prime$ and $s\in S$,
\begin{align*}
\norm{\rba(\e)s-s}&\leq\norm{\rba(\e)(s-x_{i_0})} + \norm{\rba(\e)x_{i_0}-x_{i_0}}+\norm{x_{i_0}-s}\\
& < \norm{\rba}\cdot M\cdot \epsilon/(3\norm{\rba}M) + \epsilon/3 + \epsilon/3\\
&=\epsilon.
\end{align*}
\end{proof}

Quite the opposite of the situation in Lemma~\ref{lem:totally_bounded_subsets},  subsets $S$ with the property  that $\lim_{\lambda}\rba(\e)s=s$ uniformly on $S$ can also be unbounded. This is shown by the following example, which will be continued in Example~\ref{ex:second_appearance}, and which will be considered in detail in Section~\ref{sec:worked_example}.

\begin{example}\label{ex:first_appearance}
Let $\ba=\conto{\mathbb R}$,  and consider the continuous left regular representation of $\ba$.
For every integer $n\geq 1$, we choose a function $e_n\in\conto{\RR}$ that takes values in $[0,1]$, equals 1 on $[-n,n]$, and equals 0 on $(-\infty,-n-1]\cup[n+1,\infty)$. Then $\{e_n\}_{n=1}^\infty$ is easily seen to be a bounded approximate identity for $\ba$.

Choose $f_0\in\conto{\mathbb R}$ such that $f_0(t)\geq 0$ for all $t\in\RR$ and such that $\norm{f_0}>1$. Let
\[
S=\{\,f\in\pos{\conto{\mathbb R}} : f(t)\leq f_0(t)\textup{ for all }t\in\mathbb R\textup{ such that }f_0(t)\leq 1\,\}.
\]
Then $S$ is non-empty and unbounded, because it contains functions with arbitrarily large norms that have compact supports in the non-empty open set $\{\,t\in\RR : f_0(t)>1\,\}$.

We claim that $\lim_{n\to\infty}\norm{e_n f-f}=0$ uniformly for $f\in S$. Indeed, if $\epsilon>0$ is given, then we choose $n_0$ so large that both $\{\,t\in\RR : f_0(t)>1\,\}\subset[-n_0,n_0]$ and $0\leq f_0(t)<\epsilon/2$ for all $t$ such that $|t|\geq n_0$.
Let $n\geq n_0$ and let $f\in S$.  If $|t|\leq n_0$, then $|t|\leq n$, so we have $|e_n(t)f(t)-f(t)|=|f(t)-f(t)|=0$. If $|t|> n_0$, then $f_0(t)\leq 1$, so, using the definition of $S$, we see that $0\leq f(t)\leq f_0(t) <\epsilon/2$. This implies that $|e_n(t)f(t)-f(t)|\leq 2 f_0(t) <\epsilon$. Therefore, $\norm{e_n f-f}<\epsilon$ for all $n\geq n_0$ and $f\in S$, and our claim has been established.
\end{example}

One might wonder to which extent the property that $\lim_{\lambda\in\Lambda}\rba(\e)s=s$ uniformly on $S$ depends on the particular choice of the bounded left approximate identity $\ai$. If $S$ is bounded, then it does not: it is an intrinsic property of $S$. This is implied by the following result.

\begin{lemma}\label{lem:transfer_for_bounded_sets}
Let $\ba$ be a normed algebra, let $\{e^\prime_\mu\}_{\mu\in M}$ be a net in $\ba$, let $\bs$ be a normed space, let $\rba$ be a continuous representation of $\ba$ on $\bs$, and let $S$ be a bounded non-empty subset of $\bs$ such that $\lim_\mu \rba(e^\prime_\mu)s=s$ uniformly on $S$.
Suppose that $\ai$ is a bounded left approximate identity for $A$. Then also $\lim_\lambda \rba(\e)s=s$ uniformly on $S$.

\begin{proof}
Let $\epsilon>0$ be given. Using the boundedness of $\ai$, we can choose $\mu_0\in M$ such that both $\norm{\rba}(\sup_{\lambda\in\Lambda}\norm{\e})\norm{\rba(e^\prime_{\mu_0})s-s}<\epsilon/3$ and $\norm{\rba(e^\prime_{\mu_0})s-s}<\epsilon/3$ for all $s\in S$. Using the boundedness of $S$, we see that there exists $\lambda_0$ such that $\norm{\rba}\norm{\e e^\prime_{\mu_0}-e^\prime_{\mu_0}}(\sup_{s\in S}\norm{s})<\epsilon/3$ for all $\lambda\geq\lambda_0$. Then, for all $\lambda\geq\lambda_0$ and $s\in S$,
\begin{align*}
\norm{\rba(\e)s-s}&\leq\norm{\rba(\e-\e e^\prime_{\mu_0})s}+\norm{\rba(\e e^\prime_{\mu_0} -e^\prime_{\mu_0})s}+\norm{\rba(e^\prime_{\mu_0})s-s}\\
&\leq \norm{\rba}\norm{\e}\norm{s-\rba(e^\prime_{\mu_0})s} + \norm{\rba}\norm{\e e^\prime_{\mu_0}-e^\prime_{\mu_0}}\norm{s} + \epsilon/3\\
&<\epsilon/3 + \epsilon/3 + \epsilon/3\\
& =\epsilon.
\end{align*}
\end{proof}
\end{lemma}

It would be very nice if a similar result were true for unbounded $S$, for the following reason. It is known, as a consequence of Sinclair's work on analytic semigroups in Banach algebras, that every separable Banach algebra with a bounded two-sided approximate identity, as well as  every Banach algebra with a sequential bounded two-sided approximate identity, has a commutative bounded two-sided sequential approximate identity (even one that is bounded by 1 in an equivalent algebra norm); see \cite[Corollary~2.9.43]{Dalesbook}, \cite[Theorem~3.5]{doran}, or \cite[Corollary~5.3.4]{Palmer_I}. Consequently, commutative bounded (left) approximate identities are in no short supply. A result similar to Lemma~\ref{lem:transfer_for_bounded_sets} for unbounded $S$ would, therefore, allow us in a considerable number of cases to transfer the uniform convergence for a given bounded left approximate identity to such a commutative bounded (left) approximate identity, and subsequently Theorem~\ref{thm:simultaneous_power_factorization} could then be applied.
 All in all,  a result for unbounded $S$ similar to Lemma~\ref{lem:transfer_for_bounded_sets}  would imply that the conclusions of Theorem~\ref{thm:simultaneous_power_factorization} would hold in quite a few cases, regardless of the original $\ai$ being commutative or $S$ being bounded.

It is, therefore, relevant to note that it can actually occur that $\lim_{\lambda\in\Lambda}\rba(\e)s=s$ uniformly on $S$ for one bounded left approximate identity $\ai$ of $\ba$, whereas this fails for another bounded left approximate identity. This is already possible for a commutative algebra, as is shown in the following example, which is a continuation of Example~\ref{ex:first_appearance}.

\begin{example}\label{ex:second_appearance}
We consider the left regular representation of $\conto{\RR}$ and the unbounded subset $S$  as in Example~\ref{ex:first_appearance} again. For every integer $n\geq 1$, we choose a function $e^\prime_n\in\conto{\RR}$ that takes values in $[0,1]$, equals $1-1/n$ on $[-n,n]$, and equals 0 on $(-\infty,-n-1]\cup[n+1,\infty)$. We claim that $\{e^\prime_n\}_{n=1}^\infty$ is a bounded approximate identity for $\conto{\RR}$. To see this, let $\epsilon>0$ and $f\in\conto{\RR}$ be given. We choose $n_0\geq 1$ such that both $\norm{f}/n_0<\epsilon$ and $|f(t)|<\epsilon/2$ for all $t$ such that $|t|\geq n_0$.  Now let $n\geq n_0$. If $|t|\leq n$, then $|e^\prime_n(t)f(t)-f(t)|=|f(t)/n|\leq\norm{f}/n\leq\norm{f}/n_0<\epsilon$.  If $|t|>n$, then $|t|>n_0$, and in that case $|e^\prime_n(t)f(t)-f(t)|\leq 2 |f(t)|<\epsilon$. This shows that $\norm{e^\prime_n f-f}<\epsilon$ for all $n\geq n_0$, and our claim has been established.

However, it is not true that $\lim_{\lambda\in\Lambda}\rba(e^\prime_n)s=s$ uniformly on $S$. In fact, it is even true that $\sup_{f\in S}\norm{e^\prime_n f-f}=\infty$ for all sufficiently large $n$. To see this, we choose $n_0$ such that $\{\,t\in\RR : f_0(t)>1\,\}\subset[-n_0,n_0]$.
Let $n\geq n_0$. If $f\in \contc{\RR}$ is supported in $\{t\in\RR : f_0(t)>1\,\}$, then $f\in S$. Furthermore, $\norm{e^\prime_n f-f}=\sup_{|t|\leq n_0}|e^\prime_n(t)f(t)-f(t)|=\sup_{|t|\leq n_0} |f(t)|/n=\norm{f}/n$. Since $\norm{f}$ can be arbitrarily large, this shows that $\sup_{f\in S}\norm{e^\prime_n f-f}=\infty$ for all $n\geq n_0$. Consequently, the convergence on $S$ using the $e^\prime_n$ is pointwise, but not uniform.
\end{example}

Our next two results, Lemmas~\ref{lem:uniform_convergence_is_preserved_under_the_analytic_function} and its Corollary~\ref{cor:both properties_are_preserved_under_the_analytic_function}, will be applied only in the context of Proposition~\ref{prop:specific_choice}. In the literature, the pertinent statements in that proposition are proved in that particular context, but the underlying phenomenon is more general. It seems worthwhile to make it explicit.

\begin{lemma}\label{lem:uniform_convergence_is_preserved_under_the_analytic_function} Let $\uniba$ be a unital Banach algebra, let $\ai$ be a bounded net in $\uniba$ of bound $M\geq 1$, let $\eta>0$, and
suppose that $f:\{\,z\in\CC\colon \abs{z}<M+\eta\,\}\rightarrow \CC$ is analytic with $f(1)=1$. Then $\fai$ is a bounded net in $\uniba$, and, for all $\lambda\in\Lambda$, $\fe$  is an element of the unital Banach subalgebra of $\uniba$ that is generated by $\e$.

Furthermore, if $\ns$ is a normed space, if $\rba$ is a continuous unital representation of $\uniba$ on $\ns$, and if $S$ is a non-empty subset of $\ns$ such that $\lim_\lambda \rba(\e)s=s$ uniformly on $S$, then $\lim_\lambda \rba(\fe) s=s$ uniformly on $S$.
\end{lemma}

In Lemma~\ref{lem:uniform_convergence_is_preserved_under_the_analytic_function}, $\fe$ is defined using the Maclaurin series of $f$. If $\FF=\CC$, then this agrees with the holomorphic functional calculus, since the spectrum of $\e$ is contained in the domain of $f$ for all $\lambda\in\Lambda$. If $\FF=\RR$, then it is tacitly assumed that $f(z)\in\RR$ if $z\in\RR$ and $|z|<M+\eta$. The same remarks apply to Corollary~\ref{cor:both properties_are_preserved_under_the_analytic_function} below.

\begin{proof}[Proof of Lemma~\ref{lem:uniform_convergence_is_preserved_under_the_analytic_function}]
Let $f(z)=\sum_{n=0}^{\infty}\alpha_nz^n$ be the Maclaurin series of $f$, which is absolutely convergent if $\abs{z}<M+\eta$. It is clear that $\norm{\fe}\leq\sum_{n=0}^\infty |\alpha_n|M^n$ for all $\lambda\in\Lambda$, so that
$\fai$ is a bounded net in $\uniba$.  It is likewise clear that, for all $\lambda\in\Lambda$, $\fe$ is an element of the unital Banach subalgebra of $\uniba$ that is generated by $\e$. Since $\sum_{n=0}^{\infty}\alpha_n=f(1)=1$, we have, for all $s\in
S$,
\begin{align*}
\norm{\rba(\fe)s-s} &=\norm{\sum_{n=0}^\infty \alpha_n \rba(\e^n)
s-s}\\
&=\norm{\sum_{n=0}^\infty \alpha_n \rba(\e^n) s-\sum_{n=0}^\infty
\alpha_n s }\\
&\leq\sum_{n=0}^\infty\abs{\alpha_n}\norm{\rba(\e^n-\unit) s}\\
&=\sum_{n=1}^\infty\abs{\alpha_n}\norm{\rba((\e^{n-1}+...+\e+\unit)(\e-\unit))s}\\
&\leq \norm{\rba}\sum_{n=1}^\infty\abs{\alpha_n}(M^{n-1}+...+M+1)\norm{\rba(\e-\unit) s}\\
&\leq\norm{\rba}\left(\sum_{n=1}^\infty n \abs{\alpha_n}M^{n-1}\right)\norm{\rba(\e)
s-s},
\end{align*}
where the fact that $M\geq 1$ was used in the final step.
Since the Maclaurin series of $f^\prime$ is absolutely convergent in $M$, we see that $\sum_{n=1}^{\infty}n\abs{\alpha_n}M^{n-1}<\infty$.
Combining this with the assumption that $\lim_\lambda \rba(\e) s=s$ uniformly on $S$, the statement in the
lemma follows.
\end{proof}

\begin{corollary}\label{cor:both properties_are_preserved_under_the_analytic_function} Let $\ba$ be a Banach  subalgebra of the unital Banach algebra $\uniba$, and suppose that $\ba$ contains a left approximate identity
$\ai$ for itself of bound $M\geq 1$. Let $\eta>0$, and suppose that $f:\{\,z\in\CC\colon \abs{z}<M+\eta\,\}\rightarrow \CC$ is analytic with $f(1)=1$. Then $\fai$ is a bounded net in $\uniba$ such that $\lim_{\lambda}\fe a=a$ for all $a\in\ba$. For all $\lambda\in\Lambda$, $\fe$ is an element of the unital Banach subalgebra of $\uniba$ that is generated by $\e$.

Furthermore, if $\ns$ is a normed space, if $\rba$ is a continuous unital representation of $\uniba$ on $\ns$, and if $S$ is a non-empty subset of $\ns$ such that $\lim_\lambda \rba(\e)s=s$ uniformly on $S$, then $\lim_\lambda \rba(\fe) s=s$ uniformly on $S$.
\end{corollary}

If $\ba\neq\{\,0\,\}$, then it is automatic that $M\geq 1$. The requirement that $M\geq 1$ is necessary to be able to include the case of the zero Banach subalgebra, because also in that case one needs $f(1)$ to be defined in the statement of the corollary. The same remark applies to several results in the sequel.

\begin{proof}[Proof of Corollary~\ref{cor:both properties_are_preserved_under_the_analytic_function}]
We apply Lemma~\ref{lem:uniform_convergence_is_preserved_under_the_analytic_function} in two contexts: first in that of the left regular representation of $\uniba$, where  we take all singleton subsets of $\ba$ for $S$, and then in that of the given representation $\rba$ of $\uniba$ on $\ns$.
\end{proof}

Our next result is also concerned with preservation of uniform convergence. It will be needed in the proof of Proposition~\ref{prop:simultaneous_power_factorization_for_bounded sets}.

\begin{lemma}\label{lem:uniform_convergence_is_preserved_under_action}
Let $\ba$ be a normed subalgebra of the normed algebra $\uniba$, and suppose that $\ba$ contains a bounded left approximate identity $\ai$ for itself. Let $\ns$ be a normed space, and let $\rba$ be a continuous representation of $\uniba$ on $\ns$. Suppose that $S$ is a bounded non-empty subset of $\ns$ such that $\lim_\lambda \rba(\e)s=s$ uniformly on $S$, and that $b\in\uniba$ is such that $b\e\in\ba$ for all $\lambda\in \Lambda$.  Then $\lim_\lambda \rba(\e) \rba(b)s=\rba(b)s$ uniformly on $S$.
\end{lemma}

\begin{proof} Let $\epsilon>0$ be given. Choose a bound $M\geq 1$ for $\ai$ and a bound $M^\prime$ for $S$. Choose $\lambda_0\in\Lambda$ such that $\norm{\rba}M\norm{b}\norm{\rba(e_{\lambda_0})s-s}<\epsilon/3$ for all $s\in S$. There exists $\lambda^\prime\in\Lambda$ such that $\norm{\rba}\norm{\e be_{\lambda_0}-be_{\lambda_0}}M^\prime<\epsilon/3$ for all $\lambda\geq\lambda^\prime$, since $be_{\lambda_0}\in\na$. Then, for all $\lambda\geq\lambda^\prime$ and $s\in S$, we have
\begin{align*}
\norm{\rba(\e)\rba(b)s-\rba(b)s}&=\norm{\rba(\e b-b)s}\\
&\leq\norm {\rba( \e b - \e b e_{\lambda_0})s}\! + \! \norm{\rba(\e b e_{\lambda_0} - b e_{\lambda_0} )s} \! + \! \norm{\rba(b e_{\lambda_0} - b )s}\\
&\leq \norm{\rba}M\norm{b}\norm{s-\rba(e_{\lambda_0})s} + \norm{\rba}\norm{\e b e_{\lambda_0}- b e_{\lambda_0}}M^\prime + \\
&\quad + \norm{\rba}\norm{b}\norm{\rba(e_{\lambda_0})s-s}\\
&<\epsilon/3 + \epsilon/3 + \epsilon/3\\
&=\epsilon.
\end{align*}

\end{proof}

Finally, we arrive at the main proposition in the current section. It will play a key role in the sequel. The use of the meromorphic function $z\mapsto(1-r+rz)^{-1}$ is a common occurrence in the literature on factorization theorems, and it goes back to Cohen's original paper \cite{cohen1959}.

\begin{proposition}\label{prop:specific_choice} Let $\ba$ be a Banach subalgebra of the unital Banach algebra $\uniba$, and suppose that $\ba$ contains a left approximate identity $\ai$ for itself of bound $M\geq 1$.  Choose $r$ such that $0<r<(M+1)^{-1}$, and let
\[
f(z)=\frac{1}{1-r+rz}\quad\left(z\in\CC,\, |z|<\frac{1}{r} -1\right).
\]
Then $\fai$ is a bounded net in $\uniba$ such that $\lim_{\lambda}f(\e)a=a$ for all $a\in\ba$. For every $\lambda\in\Lambda$, $f(\e)$  is an element of the unital Banach subalgebra of $\uniba$ that is generated by $\e$, 
\begin{equation}\label{eq:upper_bound_for_f_e_lambda}
\norm{f(\e)}\leq(1-r-rM)^{-1},
\end{equation}
$f(\e)$ is invertible in $\uniba$, and 
\begin{equation}\label{eq:inverse_of_f_e_lambda}
f(\e)^{-1}=(1-r)\unit + r\e.
\end{equation}

Furthermore, if $\ns$ is a normed space, if $\rba$ is a continuous unital representation of $\uniba$ on $\ns$, and if $S$ is a non-empty subset of $\bs$ such that $\lim_\lambda\rba(\e)s=s$ uniformly on $S$, then, for all $j\geq 1$, $\lim_{\lambda}\rba(f(\e)^j)s=s$ uniformly on $S$. If, in addition, $S$ is bounded, and if $b\in\uniba$ is such that $b\e\in\ba $ for all $\lambda\in\Lambda$, then, for all $j\geq 1$, $\lim_{\lambda}\rba(f(\e)^j)\rba(b)s=\rba(b)s$ uniformly on $S$.
\end{proposition}

\begin{proof}
The pole of $f$ in the complex plane is located at $1/r - 1$. Since $|1/r-1|=1/r-1>M$, Corollary~\ref{cor:both properties_are_preserved_under_the_analytic_function} applies. This shows that $\fai$ is a bounded net in $\uniba$ such that $\lim_{\lambda}f(\e)a=a$ for all $a\in\ba$, and that, for every $\lambda\in\Lambda$, $f(\e)$ is an element of the unital Banach subalgebra of $\uniba$ that is generated by $\e$.

Furthermore, for $z\in\CC$ such that $|z|<1/r-1$, we have
\[
f(z)=\frac{1}{1-r}\sum_{n=0}^\infty\left(\frac{r}{r-1}\right)^n
z^n,
\]
so that, for all $\lambda\in\Lambda$,
\[
\norm{f(\e)}\leq\frac{1}{1-r}\sum_{n=0}^{\infty}\left(\frac{r}{1-r}\right)^nM^n=(1-r-rM)^{-1}.
\]
Since $f$ has no zero on its domain, it is clear from the properties of the functional calculus that all $f(\e)$ are invertible in $\uniba$ with inverses as in \eqref{eq:inverse_of_f_e_lambda}.

Corollary~\ref{cor:both properties_are_preserved_under_the_analytic_function} applies to $f^j$ for all $j\geq 1$, and this shows that, for all $j\geq 1$, $\lim_{\lambda}\rba(f(\e)^j)s=\lim_{\lambda}\rba(f^j(\e))s=s$ uniformly on $S$ whenever $\lim_\lambda\rba(\e)s=s$ uniformly on $S$.

If $S$ is bounded, if $\lim_\lambda\rba(\e)s=s$ uniformly on $S$, and if $b\in\uniba$ is such that $b\e\in\ba$ for all $\lambda\in\Lambda$, then Lemma~\ref{lem:uniform_convergence_is_preserved_under_action} shows that  $\lim_{\lambda}\rba(\e)s^\prime=s^\prime$ uniformly on $S^\prime:=\rba(b)S$. Applying what we have just proved to $S^\prime$, we conclude that, for all $j\geq 1$,  $\lim_\lambda \rba(f(\e)^j) \rba(b)s=\rba(b)s$ uniformly on $S$.  This completes the proof.
\end{proof}

\section{simultaneous power factorization}\label{sec:simultaneous power factorization}

This section contains the central result of the paper, which is Theorem~\ref{thm:simultaneous_power_factorization}.  In that theorem, it is assumed that the subset $S$ is bounded or that the bounded left approximate identity $\ai$ of $\ba$ is commutative. The conclusions of the theorem are the same in both cases. In fact, the proof of the theorem covers both cases at the same time, because it relies on the identical parts (1) through (6) in Proposition~\ref{prop:simultaneous_power_factorization_for_bounded sets} (for the case of bounded $S$) and Proposition~\ref{prop:simultaneous_power_factorization_for_commutative_approximate identity} (for the case of commutative $\ai$). The proofs of the two propositions, however, are different. We shall now establish these two preparatory results, and we start with the bounded case. For this, we first record an algebraic identity, which was used implicitly in \cite[proof of Theorem~1]{Allan1976}. The elementary proof is omitted.

\begin{lemma}\label{lem:ring_identity}
Let $R$ be a unital ring. Then, for all $a,b\in R$ and $n\geq1$,
\[a^n-b^n=\sum_{i=0}^{n-1}a^{n-1-i}(a-b)b^{i}.\]
\end{lemma}

%
%

\begin{proposition}\label{prop:simultaneous_power_factorization_for_bounded sets}
Let $\ba$ be a Banach subalgebra of the unital Banach algebra $\uniba$, and suppose that $\ba$ contains a left approximate identity $\ai$ for itself of bound $M\geq 1$.  Let $\ns$ be a normed space, and let $\rba$ be a continuous unital representation of
$\uniba$ on $X$. Suppose that $S$ is a bounded non-empty subset of $\ns$ such that $\lim_\lambda \rba(\e)s=s$ uniformly on $S$.

Choose $r$ such that $0<r<(M+1)^{-1}$, and put $\Delta=(1-r-rM)^{-1}+1>1$.

Then, for every $\epsilon>0$ and for every sequence $\seqk{j}$ of strictly positive integers, there exist sequences $\seqko{b}$ in $\uniba$ and $\{u_k\}_{k=1}^\infty$ in $\bigcup_{\lambda\in\Lambda}\{\,\e\,\}$ such that
\begin{enumerate}
\item $b_0=\unit$, and $b_k\in\inv{\uniba}$ for all $k\geq 0$;
\item $\norm{b_k^{-1}}\leq \Delta^k$ for all $k\geq 0$;
\item $\norm{\rba(b_k^{-j})s-\rba(b_{k-1}^{-j})s}<\frac{\epsilon}{2^k}$ for all $k\geq 1$, $j=1,...,j_k$, and $s\in S$;
\item $b_k=(1-r)^k\unit+\sum_{i=1}^k r(1-r)^{i-1}u_i$ for all $k\geq 0$;
\item for all $k\geq 0$, $b_k^{-1}$ is an element of the unital Banach subalgebra of $\uniba$ that is generated by $\{\,u_1,\ldots,u_k\,\}$;
\item there exists a chain $\lambda_1\leq\lambda_2\leq\lambda_3\leq\ldots$ in $\Lambda$ such that $u_k=e_{\lambda_k}$ for all $k\geq 1$. If $\Lambda$ does not have a largest element, one can require that $\lambda_1<\lambda_2<\lambda_3<\ldots$.
\end{enumerate}

\end{proposition}

\begin{proof}

We shall use an inductive procedure to construct sequences $\seqko{b}$ and $\{u_k\}_{k=1}^\infty$ satisfying (1) through (6). During this construction, part (6) is then to be interpreted as a requirement for all indices that have been considered so far.

As a preparation, we apply Proposition~\ref{prop:specific_choice} with our chosen $r$. The properties of the net $\fai$ in $\uniba$ that is now available will be used repeatedly in the current proof.

We start the induction with $k=0$. In that case, we need to find only $b_0$, and we choose $b_0=\unit$. Clearly the parts (1), (2), (4), and (5) are then satisfied; the parts (3) and (6) are not applicable for $k=0$.

We turn to $k=1$. Proposition~\ref{prop:specific_choice} asserts that, for all $j\geq 1$,  $\lim_{\lambda}\rba(f(\e)^j)s=s$ uniformly on $S$. Since, for $k=1$ (in fact, for each $k\geq 1$), part (3) involves only finitely many values of $j$, there exists $\lambda_1\in\Lambda$ such that $\norm{\rba(f(e_{\lambda_1})^j)s-s}<\epsilon/2$ for all $j=1,\ldots,j_1$ and $s\in S$. We choose $b_1=f(e_{\lambda_1})^{-1}$ and $u_1=e_{\lambda_1}$. Since $b_0=\unit$, part (3) is now clear. Part (1) is obviously satisfied, and part (2) follows from \eqref{eq:upper_bound_for_f_e_lambda}, which yields that even $\norm{b_1^{-1}}\leq (1-r-rM)^{-1}=\Delta-1$. Since \eqref{eq:inverse_of_f_e_lambda} shows that $b_1 =(1-r)\unit + e_{\lambda_1}$, part(4) is satisfied. Proposition~\ref{prop:specific_choice} yields part (5), and part (6) is trivially satisfied when only one index greater than or equal to 1  has been considered thus far.

With an eye towards the proof of Proposition~\ref{prop:simultaneous_power_factorization_for_commutative_approximate identity} below, we note that the boundedness of $S$ has not been used so far.

We now assume that $k\geq 1$, and that $b_0,\ldots, b_k$ and $u_1,...,u_k$ have been defined such that (1) through (5) hold, and such that (6) holds for all $k_1, k_2$ such that $1\leq k_1\leq k_2\leq k$. This is true for $k=0$ and $k=1$. We shall proceed to find $b_{k+1}$ and $u_{k+1}$. For this, we need a few preparations.

For $\lambda\in\Lambda$, we define
\begin{equation}\label{eq:g_definition}
g(\lambda)=(1-r)^k\unit+f(\e)\sum_{i=1}^k r(1-r)^{i-1}u_i.
\end{equation}
Using part (4) of the induction hypothesis, we see that
\begin{align}\label{eq:g_lambda_close_to_b_k}
\begin{split}
\norm{g(\lambda)-b_k}
&=\norm{\sum_{i=1}^k r(1-r)^{i-1}(f(\e)u_i-u_i)}\\
&\leq\sum_{i=1}^k r(1-r)^{i-1}\norm{f(\e)u_i-u_i}.
\end{split}
\end{align}

Since there are only finitely many values of $i$ in the summation in \eqref{eq:g_lambda_close_to_b_k}, and since $b_k\in\inv{\uniba}$ by part (1) of the inductive hypothesis, we conclude from Proposition~\ref{prop:specific_choice}, using that $\inv{\uniba}$ is open and that the inversion is continuous on $\inv{\uniba}$, that there exists $\lambda^\prime\in\Lambda$ such that both $g(\lambda)\in \inv{\uniba}$ and
\begin{equation}\label{eq:g_lambda_inverse_upper_bound}
\norm{g(\lambda)^{-1}}\leq\norm{b_k^{-1}}+1
\end{equation}
for all $\lambda\geq\lambda^\prime$. Moreover, this can be so arranged that, for $\lambda\geq\lambda^\prime$, $g(\lambda)^{-1}$ can be expressed as a Neumann series, making it clear that it is an element of the unital Banach subalgebra of $\uniba$ that is generated by $b_k^{-1}$ and $g(\lambda)$. Part (5) of the induction hypothesis, together with Proposition~\ref{prop:specific_choice}, then shows that $g(\lambda)^{-1}$ is an element of the unital Banach subalgebra of $\uniba$ that is generated by $\{\,u_1,\ldots, u_k,\e\}$.

For $\lambda\in\Lambda$, we define
\[
\bl=(1-r)^{k+1}\unit+\sum_{i=1}^k r(1-r)^{i-1} u_i+r(1-r)^{k} \e.
\]
Since $f(\e)^{-1}=(1-r)\unit+r\e$ by \eqref{eq:inverse_of_f_e_lambda}, one sees easily that, for all $\lambda\in\Lambda$,
\begin{equation}\label{eq:b_lambda_as_a_product}
\bl=f(\e)^{-1}g(\lambda).
\end{equation}
Therefore, if $\lambda\geq\lambda^\prime$, then $\bl\in\inv{\uniba}$. Since then $\bl^{-1}=g(\lambda)^{-1}f(\e)$, we see from the corresponding statement for $g(\lambda)^{-1}$ and Proposition~\ref{prop:specific_choice} that, for all $\lambda\geq\lambda^\prime$, $\bl^{-1}$ is an element of the unital Banach subalgebra of $\uniba$ that is generated by $\{\,u_1,\ldots, u_k,\e\,\}$.
Furthermore, using \eqref{eq:upper_bound_for_f_e_lambda} and part (2) of the induction hypothesis, we have, for $\lambda\geq\lambda^\prime$,
\begin{align}\label{eq:b_lambda_inverse_upper_bound}
\begin{split}
\norm{\bl^{-1}}&\leq\norm{g(\lambda)^{-1}}\norm{f(\e)}\\
&\leq (\norm{b_k^{-1}}+1)\cdot \frac{1}{1-r-rM}\\
&\leq (\Delta^k+1)\cdot\frac{1}{1-r-rM}\\
&<\frac{1}{1-r-rM}\Delta^k + \Delta\\
&< \frac{1}{1-r-rM}\Delta^k + \Delta^k\\
&=\Delta^{k+1}.
\end{split}
\end{align}

Continuing our preparations,  using Lemma~\ref{lem:ring_identity}, \eqref{eq:b_lambda_inverse_upper_bound}, and part (2) of the induction hypothesis, we see that, for all $j=1,\ldots,j_{k+1}$, $\lambda\geq\lambda^\prime$, and $s\in S$,

\begin{align}\label{eq:application_of_ring_lemma}
\begin{split}
\norm{\rba(\bl^{-j})s-\rba(b_k^{-j})s}&=\norm{\sum_{i=0}^{j-1}\rba\left(\bl^{-(j-1-i)}\right )\rba\left([\bl^{-1} -  b_k^{-1}]b_k^{-i}\right)s}\\
&\leq \norm{\rba}\sum_{i=0}^{j-1}\norm{\bl^{-1}}^{j-1-i}\norm{\rba\left([\bl^{-1} -  b_k^{-1}]b_k^{-i}\right)s} \\
&\leq \norm{\rba}\sum_{i=0}^{j-1}\Delta^{(k+1)(j-1-i)}\norm{\rba\left([\bl^{-1} -  b_k^{-1}]b_k^{-i}\right)s}\\
&\leq \norm{\rba}\sum_{i=0}^{j-1}\Delta^{(k+1)(j-1)}\norm{\rba\left([\bl^{-1} -  b_k^{-1}]b_k^{-i}\right)s}\\
&\leq \norm{\rba}\sum_{i=0}^{j-1}\Delta^{(k+1)(j_{k+1}-1)}\norm{\rba\left([\bl^{-1} -  b_k^{-1}]b_k^{-i}\right)s}\\
&\leq \norm{\rba}\Delta^{(k+1)(j_{k+1}-1)}\sum_{i=0}^{j_{k+1}-1}\norm{\rba\left([\bl^{-1} -  b_k^{-1}]b_k^{-i}\right)s}.
\end{split}
\end{align}
Furthermore, if $\lambda\geq\lambda^\prime$, $i=0,\ldots,j_{k+1}-1$, and $s\in S$, then, using \eqref{eq:b_lambda_as_a_product} and \eqref{eq:g_lambda_inverse_upper_bound}, we have

\begin{align}\label{eq:key_to_uniform_convergence}
\begin{split}
\Vert&\rba\bigl([\bl^{-1}  -  b_k^{-1}]b_k^{-i}\bigr)s\Vert=\norm{\rba\left([g(\lambda)^{-1}f(\e)-b_k^{-1}]b_k^{-i}\right)s}\\
&=\norm{\rba\left(g(\lambda)^{-1}f(\e)b_k^{-i}-g(\lambda)^{-1}b_k^{-i}+g(\lambda)^{-1}b_k^{-i}-b_k^{-1}b_k^{-i}\right)s}\\
&\leq \norm{\rba(g(\lambda)^{-1})}\norm{\rba(f(\e))\rba(b_k^{-i})s-\rba(b_k^{-i})s} + \norm{\rba(g(\lambda)^{-1}-b_k^{-1})\rba(b_k^{-i})s}\\
&\leq \norm{\rba}(\norm{b_k^{-1}}+1)\norm{\rba(f(\e))\rba(b_k^{-i})s-\rba(b_k^{-i})s} + \norm{\rba(g(\lambda)^{-1}-b_k^{-1})}\cdot K,
\end{split}
\end{align}
where
\[
K=\sup\{\,\norm{\rba(b_k^{-i})s} : i=0,\ldots,j_{k+1}-1,\,s\in S\,\}.
\]
We note that $K<\infty$ since $S$ is bounded.

It follows from \eqref{eq:g_lambda_close_to_b_k}, Proposition~\ref{prop:specific_choice}, and the continuity of the inversion on $\inv{\uniba}$ that
\begin{equation}\label{eq:first_term_bounded_S}
\lim_{\lambda\geq\lambda^\prime} \norm{\rba(g(\lambda)^{-1}-b_k^{-1})}\cdot K=0.
\end{equation}
Furthermore, part (5) of the induction hypothesis implies that $b_k^{-i}\e\in A$ for all $i=0, \ldots,j_{k+1}-1$ and $\lambda\in\Lambda$. Since $S$ is bounded, Proposition~\ref{prop:specific_choice} then shows that, for all $i=0,\ldots,j_{k+1}-1$,
\begin{equation}\label{eq:second_term_bounded_S}
\lim_\lambda \norm{\rba(f(\e))\rba(b_k^{-i})s-\rba(b_k^{-i})s} =0
\end{equation}
uniformly on $S$. It is now clear from \eqref{eq:key_to_uniform_convergence}, \eqref{eq:first_term_bounded_S}, and \eqref{eq:second_term_bounded_S} that, for all $i=0,\ldots,j_{k+1}-1$,
\[
\lim_{\lambda\geq\lambda^\prime}\norm{\rba\bigl([\bl^{-1}  -  b_k^{-1}]b_k^{-i}\bigr)s}=0
\]
uniformly on $S$.
Finally, \eqref{eq:application_of_ring_lemma} then shows that, for all $j=1,\ldots,j_{k+1}$,
\[
\lim_{\lambda\geq\lambda^\prime}\norm{\rba(\bl^{-j})s-\rba(b_k^{-j})s}=0
\]
uniformly on $S$. In particular, there exists $\lambda^{\prime\prime}\geq\lambda^\prime$ such that
\begin{equation*}
\norm{\rba(b(\lambda^{\prime\prime})^{-j})s-\rba(b_k^{-j})s}\leq\frac{\epsilon}{2^{k+1}}
\end{equation*}
for all $j=1,\ldots,j_{k+1}$ and $s\in S$. It is clear that $\lambda^{\prime\prime}$ can also be chosen such that, in addition, $\lambda^{\prime\prime}\geq \lambda_k$, or, if $\Lambda$ does not have a largest element, such that $\lambda^{\prime\prime}>\lambda_k$.
The induction step in the construction is then completed by choosing $b_{k+1}=b(\lambda^{\prime\prime})$ and $u_{k+1}=e_{\lambda^{\prime\prime}}$, and where the chain under part (6) is extended by adding $\lambda_{k+1}:=\lambda^{\prime\prime}$.

\end{proof}

As announced, the next result, is almost identical to Proposition~\ref{prop:simultaneous_power_factorization_for_bounded sets}. The only two differences are that $\ai$ is required to be commutative, but that $S$ need not be bounded.

\begin{proposition}\label{prop:simultaneous_power_factorization_for_commutative_approximate identity}

Let $\ba$ be a Banach subalgebra of the unital Banach algebra $\uniba$, and suppose that $\ba$ contains a commutative left approximate identity $\ai$ for itself of bound $M\geq 1$.  Let $\ns$ be a normed space, and let $\rba$ be a continuous unital representation of
$\uniba$ on $X$. Suppose that $S$ is a non-empty subset of $\ns$ such that $\lim_\lambda \rba(\e)s=s$ uniformly on $S$.

Choose $r$ such that $0<r<(M+1)^{-1}$, and put $\Delta=(1-r-rM)^{-1}+1>1$.

Then, for every $\epsilon>0$ and for every sequence $\seqk{j}$ of strictly positive integers, there exist sequences $\seqko{b}$ in $\uniba$ and $\{u_k\}_{k=1}^\infty$ in $\bigcup_{\lambda\in\Lambda}\{\,\e\,\}$ such that
\begin{enumerate}
\item $b_0=\unit$, and $b_k\in\inv{\uniba}$ for all $k\geq 0$;
\item $\norm{b_k^{-1}}\leq \Delta^k$ for all $k\geq 0$;
\item $\norm{\rba(b_k^{-j})s-\rba(b_{k-1}^{-j})s}<\frac{\epsilon}{2^k}$ for all $k\geq 1$, $j=1,...,j_k$, and $s\in S$;
\item $b_k=(1-r)^k\unit+\sum_{i=1}^k r(1-r)^{i-1}u_i$ for all $k\geq 0$;
\item for all $k\geq 0$, $b_k^{-1}$ is an element of the unital Banach subalgebra of $\uniba$ that is generated by $\{\,u_1,\ldots,u_k\,\}$;
\item there exists a chain $\lambda_1\leq\lambda_2\leq\lambda_3\leq\ldots$ in $\Lambda$ such that $u_k=e_{\lambda_k}$ for all $k\geq 1$. If $\Lambda$ does not have a largest element, one can require that $\lambda_1<\lambda_2<\lambda_3<\ldots$.
\end{enumerate}

\end{proposition}

\begin{proof} The proof is an adaptation of the inductive construction in the proof of Proposition~\ref{prop:simultaneous_power_factorization_for_bounded sets}.

As in that proof, we start by applying Proposition~\ref{prop:specific_choice} with our chosen $r$, and we shall work with the net $\fai$ in $\uniba$ that is then available. Note that, as a consequence of Proposition~\ref{prop:specific_choice}, $f(\e)$ and $e_{\tilde\lambda}$ commute for all $\lambda,\tilde\lambda\in\Lambda$.

Exactly as in the proof of Proposition~\ref{prop:simultaneous_power_factorization_for_bounded sets}, we can find $b_0$, $b_1$, and $u_1$ such that parts (1) through (6) are  satisfied for $k=0$ and $k=1$. Indeed, as was already remarked in that proof, for this to be possible the condition that $\lim_\lambda \rba(\e)s=s$ uniformly on $S$ is already sufficient; boundedness of $S$ is not needed.

For the induction step, we assume that $k\geq 1$, and that $b_0,\ldots, b_k$ and $u_1,...,u_k$ have been defined such that (1) through (5) hold, and such that (6) holds for all $k_1, k_2$ such that $1\leq k_1\leq k_2\leq k$. This is true for $k=0$ and $k=1$. We shall proceed to find $b_{k+1}$ and $u_{k+1}$.

As in the earlier proof, we define, for $\lambda\in\Lambda$,
\begin{equation}\label{eq:g_definition_again}
g(\lambda)=(1-r)^k\unit+f(\e)\sum_{i=1}^k r(1-r)^{i-1}u_i.
\end{equation}
As previously, we have
\begin{equation*}
\norm{g(\lambda)-b_k}\leq\sum_{i=1}^k r(1-r)^{i-1}\norm{f(\e)u_i-u_i}
\end{equation*}
for all $\lambda\in\Lambda$, and from this we conclude again that there exists  $\lambda^\prime\in\Lambda$ such that, for all $\lambda\geq\lambda^\prime$, $g(\lambda)\in \inv{\uniba}$,
\begin{equation}\label{eq:g_lambda_inverse_upper_bound_again}
\norm{g(\lambda)^{-1}}\leq\norm{b_k^{-1}}+1,
\end{equation}
and  $g(\lambda)^{-1}$ is an element of the unital Banach subalgebra of $\uniba$ that is generated by $\{\,u_1,\ldots, u_k,\e\,\}$. Note that the latter property, when combined with the commutativity of $\ai$ and part (5) of the induction hypothesis, implies that $g(\lambda)^{-1}$ and $b_k^{-1}$ commute for all $\lambda\geq\lambda^\prime$. Alternatively, and more directly, this commuting property also follows from \eqref{eq:g_definition_again}, Proposition~\ref{prop:specific_choice}, part (4) of the induction hypothesis, and the commutativity of $\ai$.

As earlier, we define, for $\lambda\in\Lambda$,
\begin{equation}\label{eq:b_definition_again}
\bl=(1-r)^{k+1}\unit+\sum_{i=1}^k r(1-r)^{i-1} u_i+r(1-r)^{k} \e.
\end{equation}
As earlier, if $\lambda\geq\lambda^\prime$, then $\bl\in\inv{\uniba}$,  $\bl^{-1}=g(\lambda)^{-1}f(\e)$,
\begin{equation}\label{eq:b_lambda_inverse_upper_bound_again}
\norm{\bl^{-1}}\leq\Delta^{k+1},
\end{equation}
and $\bl^{-1}$ is an element of the unital Banach subalgebra of $\uniba$ that is generated by $\{\,u_1,\ldots, u_k,\e\,\}$. Note that the latter property, when combined with the commutativity of $\ai$ and part (5) of the induction hypothesis, implies that $\bl^{-1}$ and $b_k^{-1}$ commute for all $\lambda\geq\lambda^\prime$. Alternatively, and more directly, this commuting property also follows from \eqref{eq:b_definition_again}, part (4) of the induction hypothesis, and the commutativity of $\ai$.

We shall now exploit the various commuting properties in the estimates that are to follow below. It is at this point that the structure of the present proof begins to differ from that of the proof of  Proposition~\ref{prop:simultaneous_power_factorization_for_commutative_approximate identity}.

For all $\lambda\geq\lambda^\prime$, $j=1,...,j_{k+1}$, and $s\in S$, we have, using
Lemma~\ref{lem:ring_identity} in the first step, the fact that $\bl^{-1}$ and $b_k^{-1}$ commute in the second step, and \eqref{eq:b_lambda_inverse_upper_bound_again} and part (2) of the induction hypothesis in the fifth step,
\begin{align}\label{eq:application_of_ring_lemma_again}
\begin{split}
\norm{\rba(\bl^{-j})s-\rba(b_k^{-j})s}
&=\norm{\rba\left(\sum_{i=0}^{j-1}\bl^{-(j-1-i)}(\bl^{-1}-b_k^{-1})b_k^{-i}\right)s}\\
&=\norm{\rba\left(\sum_{i=0}^{j-1}\bl^{-(j-1-i)}b_k^{-i}(\bl^{-1}-b_k^{-1})\right)s}\\
&\leq\sum_{i=0}^{j-1}\norm{\rba(\bl^{-(j-1-i)}b_k^{-i})}\norm{\rba(\bl^{-1}-b_k^{-1})s}\\
&\leq\sum_{i=0}^{j-1}\norm{\rba}\norm{\bl^{-1}}^{j-1-i}\norm{b_k^{-1}}^{i}\norm{\rba(\bl^{-1}-b_k^{-1})s}\\
&\leq\sum_{i=0}^{j-1}\norm{\rba}\Delta^{(k+1)j-k-1-i}\norm{\rba(\bl^{-1}-b_k^{-1})s}\\
&\leq j_{k+1}\norm{\rba}\Delta^{(k+1)j_{k+1}-k-1}\norm{\rba(\bl^{-1}-b_k^{-1})s}.
\end{split}
\end{align}
Note that, for all $\lambda\geq\lambda^\prime$ and $s\in S$,
\begin{align}\label{eq:split_into_two}
\begin{split}
\norm{\rba\left(\bl^{-1}-b_k^{-1}\right)s}&=\norm{\rba\left(g(\lambda)^{-1}f(\e)-b_k^{-1}\right)s}\\
&\leq\norm{\rba\left(g(\lambda)^{-1}(f(\e)-\unit)\right)s}+\norm{\rba(g(\lambda)^{-1}-b_k^{-1})s}.
\end{split}
\end{align}
We estimate both terms in \eqref{eq:split_into_two} separately. For $\lambda\geq\lambda^\prime$ and $s\in S$, we have, using \eqref{eq:g_lambda_inverse_upper_bound_again},
\begin{equation*}
\norm{\rba\left(g(\lambda)^{-1}(f(\e)-\unit)\right)s}\leq\norm{\pi}(\norm{b_k^{-1}}+1)\norm{\rba(f(\e))s-s}.
\end{equation*}
It then follows from Proposition~\ref{prop:specific_choice} that
\begin{equation}\label{eq:first_term_commutative_approximate_identity}
\lim_{\lambda\geq\lambda^\prime}\norm{\rba\left(g(\lambda)^{-1}(f(\e)-\unit)\right)s}=0
\end{equation}
uniformly on $S$.

The estimate for the second term in \eqref{eq:split_into_two} is slightly more involved. For $\lambda\geq\lambda^\prime$ and $s\in S$ we have, using that $g(\lambda)^{-1}$ commutes with $b_k^{-1}$ in the first step, that $f(\e)$ commutes with $u_i$ for $i=1,\ldots,k$ in the third step, and \eqref{eq:g_lambda_inverse_upper_bound_again} in the fourth step,
\begin{align}\label{eq:second_term_commutative_approximate_identity_preparation}
\begin{split}
\norm{\rba(g(\lambda)^{-1}-b_k^{-1})s}&=\norm{\rba\left(g(\lambda)^{-1}b_k^{-1}(b_k-g(\lambda))\right)s}\\
&=\norm{\rba\left(g(\lambda)^{-1} b_k^{-1}\sum_{i=1}^kr(1-r)^{i-1}(u_i-f(\e)u_i)\right)s}\\\
&=\norm{\rba\left(g(\lambda)^{-1} b_k^{-1}\sum_{i=1}^kr(1-r)^{i-1}u_i(\unit-f(\e))\right)s}\\
&\leq\norm{\rba}(\norm{b_k^{-1}}+1)\norm{b_k^{-1}\sum_{i=1}^kr(1-r)^{i-1}u_i}\norm{s-\rba(f(\e))s}.
\end{split}
\end{align}
Using Proposition~\ref{prop:specific_choice} for the last time, we can now conclude from \eqref{eq:second_term_commutative_approximate_identity_preparation} that
\begin{equation}\label{eq:second_term_commutative_approximate_identity}
\lim_{\lambda\geq\lambda^{\prime}}\norm{\rba(g(\lambda)^{-1}-b_k^{-1})s}=0
\end{equation}
uniformly on $S$. Combining \eqref{eq:split_into_two}, \eqref{eq:first_term_commutative_approximate_identity}, and \eqref{eq:second_term_commutative_approximate_identity}, we see that
\begin{equation}\label{eq:joined_into_one}
\lim_{\lambda\geq\lambda^{\prime}}\norm{\rba\left(\bl^{-1}-b_k^{-1}\right)s}=0
\end{equation}
uniformly on $S$. Finally, combining \eqref{eq:application_of_ring_lemma_again} and \eqref{eq:joined_into_one}, we conclude that, for all $j=1,\ldots,j_{k+1}$,
\[
\lim_{\lambda\geq\lambda^\prime}\norm{\rba(\bl^{-j})s-\rba(b_k^{-j})s}=0
\]
uniformly on $S$. As in the conclusion of the proof of Proposition~\ref{prop:simultaneous_power_factorization_for_bounded sets}, this allows us to find $b_{k+1}$ and $u_{k+1}$ with the required properties.
\end{proof}

We can now establish the main result of this paper. As mentioned earlier, its proof is based on the identical parts (1) through (6) of Propositions~\ref{prop:simultaneous_power_factorization_for_bounded sets} and~\ref{prop:simultaneous_power_factorization_for_commutative_approximate identity}.

\begin{theorem}\label{thm:simultaneous_power_factorization}
Let $\ba$ be a Banach algebra that has a left approximate identity $\ai$ of bound $M\geq 1$, let $\bs$ be a Banach space, and let $\rba$ be a continuous representation of $\ba$ on $\bs$.  Let $S$ be a non-empty subset of $\bs$ such that $\lim_\lambda \rba(\e)s=s$ uniformly on $S$, and suppose that $S$ is bounded or that $\ai$ is commutative.

Choose a unital  Banach superalgebra $\uniba$ of $\ba$ such that $\pi$ extends to a continuous unital representation, again denoted by $\rba$, of $\uniba$ on $\bs$.

Then, for every $\epsilon>0$, every $\delta>0$, every $r$ such that $0<r<(M+1)^{-1}$, every integer $n_0\geq 1$, and  every sequence $\{\alpha_n\}_{n=1}^\infty$ in $(1,\infty)$ such that $\lim_{n\to\infty} \alpha_n=\infty$, there exist $a\in A$ and maps $x_n:S\to \bs$ for $n\geq 1$ such that:

\begin{enumerate}
 \item $s=\rba(a^n)x_n(s)$ for all $n\geq 1$ and $s\in S$;
  \item $\norm{a}\leq M$;

  \item for all $n\geq 1$, $x_n$ is a uniformly continuous homeomorphism of $S$ onto $x_n(S)$, with the restricted map $\rba(a^n):x_n(S)\to S$ as its inverse;
  \item
 \begin{enumerate}
  \item $\norm{s-x_n(s)}\leq \epsilon$ for all $n$ such that $1\leq n\leq n_0$ and for all $s\in S$;
  \item $\norm{x_n(s)}\leq\alpha^n_n \max(\norm{s},\delta)$ for all $n\geq 1$ and $s\in S$;
  \end{enumerate}
  \item
  \begin{enumerate}
  \item if $s_1, s_2\in S$ are such that $s_1+s_2\in S$, then $x_n(s_1+s_2)=x_n(s_1)+x_n(s_2)$ for all $n\geq 1$;
  \item if $\lambda\in\mathbb F$ and $s\in S$ are such that $\lambda s\in S$, then $x_n(\lambda s)=\lambda x_n(s)$ for all $n\geq 1$;
  \end{enumerate}
  \item there exists a sequence $\{u_i\}_{i=1}^\infty$ in $\bigcup_{\lambda\in\Lambda}\{\,\e\,\}$ such that:
  \begin{enumerate}
   \item $a=\sum_{i=1}^\infty r(1-r)^{i-1}u_i$ is an element of the closed convex hull of $\{\,u_i : i\geq 1 \,\}$ in $\ba$;
   \item for every $k\geq 0$, the element $b_k=(1-r)^{k}\unit + \sum_{i=1}^{k}r(1-r)^{i-1}u_i$ of $\uniba$ is an element of the convex hull of  $\{\,\unit, u_1,\ldots,u_k\,\}$ in $\uniba$ that is invertible in $\uniba$, $b_k^{-1}$ is an element of the unital Banach subalgebra of $\uniba$ that is generated by $\{\,u_1,\ldots,u_k\,\}$, and $x_n(s)=\lim_{k\to\infty}\rba(b_k^{-n})s$ for all $n\geq 1$ and $s\in S$;
   \item there exists a chain $\lambda_1\leq\lambda_2\leq\lambda_3\leq\ldots$ in $\Lambda$ such that $u_k=e_{\lambda_k}$ for all $k\geq 1$. If $\Lambda$ does not have a largest element, one can require that $\lambda_1<\lambda_2<\lambda_3<\ldots$;
\end{enumerate}
 \item
 \begin{enumerate}
 \item if $S$ is bounded, then $x_n(S)$ is bounded for all $n\geq 1$;
 \item if $S$ is totally bounded, then $x_n(S)$ is totally bounded for all $n\geq 1$;
 \item $\lim_{\lambda}\rba(\e)x_n(s)=x_n(s)$ uniformly on $S$ for all $n\geq 1$.
  \end{enumerate}
\end{enumerate}
\end{theorem}

\begin{remark}\label{rem:Ansatz}Before setting out on the proof, let us make a few comments.
\begin{enumerate}
\item  Certainly a superalgebra $\uniba$ as in the theorem exists: the unitization of $\ba$, which is invariably used in the existing proofs of factorization theorems in the literature, is a possible choice. Hence a simultaneous power factorization also exists under the remaining hypotheses in the theorem, none of which involves $\uniba$.

There may, however, be other superalgebras satisfying the mild extension condition in the theorem. As we shall see in Section~\ref{sec:positive_simultaneous_power_factorization}, it is important to build this freedom of choice into the result. The reason is that part (6) (the only statement in which $\uniba$ figures again) is rather explicit about an actually possible form of a simultaneous power factorization. If, for a suitable choice of $\uniba$, one has additional information about the elements $b_k^{-n}$, then one has additional information about an actually possible form of a simultaneous power factorization. See Remark~\ref{rem:centralizer_algebras} for such candidate alternate superalgebras, and Theorem~\ref{thm:positive_simultaneous_power_factorization_for_banach_algebras_of_functions} for an application of the current observation.

\item With some computational perseverance, the information as provided by part (6) could lead to an explicit simultaneous power factorization in a given context. After all, one only needs to find a suitable chain $\lambda_1\leq\lambda_2\leq\lambda_3\cdots$ in $\Lambda$. We shall carry out an example of such a construction of a factorization `by hand' in Section~\ref{sec:worked_example}.

The whole proof of Theorem~\ref{thm:simultaneous_power_factorization} is, in fact, constructive, although to actually find a factorization in a concrete case one would perhaps rather start from the Ansatz as provided by part (6) than go through the estimates in the proof again.

\item Part (7) implies that the theorem can be applied repeatedly. As a consequence, if $p_1,\ldots, p_t\geq 1$ is a given set of exponents, then there exist $a_1,\ldots,a_t\in A$ and a uniformly continuous homeomorphism $x_{p_1,\ldots,p_t}:S\to x_{p_1,\ldots,p_t}(S)$ such that $s=\rba(a_1^{p_1}\cdots a_t^{p_t}) x_{p_1,\ldots,p_t}(s)$ for all $s\in S$. This is obtained by an application of the theorem to $S$ for $n=p_1$, then to $x_{p_1}(S)$ for $n=p_2$, etc.

\item
Suppose that $\Lambda$ has a largest element $\largestlambda$. Since the proof of Theorem~\ref{thm:simultaneous_power_factorization} essentially consists of repeatedly extending a chain in $\Lambda$ by a sufficiently large element of $\Lambda$, the choice $\lambda_k=\largestlambda$ for all $k\geq 1$, implying that $u_k=\elargest$ for all $k\geq 1$, must give a factorization satisfying parts (1) through (7). Indeed it does, and we shall now convince ourselves of this fact, which is still not entirely trivial.

Fist of all, it is easy to see that $\elargest$ must be a left identity element for $\ba$, and that $\elargest$ must act as the identity on $S$. Part (6a) stipulates that $a=\elargest$. An easy induction with respect to $k$ shows that, with this choice of the $u_i$, we have the factorization
\begin{equation}\label{eq:maximal_element_factorization}
b_k=(1-r)^k\unit+\sum_{i=1}^k r(1-r)^{i-1}\elargest=((1-r)\unit + r\elargest)^k
\end{equation}
for all $k\geq 1$. Since $0<r<(M+1)^{-1}$, $((1-r)\unit + r\elargest)$ is invertible in $
\uniba$. Indeed,
\begin{equation}\label{eq:neumann_series_for_inverse}
((1-r)\unit + r\elargest)^{-1}=\frac{1}{1-r}\sum_{j=0}^\infty \left(\frac{r}{r-1}\right)^j \elargest^j,
\end{equation}
where the series is absolutely convergent because $|rM/(r-1)|=rM/(1-r)<1$ since $0<r<(M+1)^{-1}$. Hence $b_k$ is also invertible in $\uniba$ for all $k\geq 1$, as it should be according to part (6b). Furthermore, it follows from \eqref{eq:neumann_series_for_inverse} that $((1-r)\unit + r\elargest)^{-1}$ acts as the identity on $S$. The same is then true for $b_k^{-n}$ for all $k\geq 1$ and $n\geq 1$, in which case part (6b) insists that $x_n(s)=\lim_{k\to\infty}\rba(b_k^{-n})s=s$ for all $n\geq 1$ and $s\in S$. Since $a$ acts as the identity on $S$, this is compatible with part (1), as it should be. A quick inspection now shows that, in fact, all properties in the parts (1) through (7) are satisfied.

The only case of true interest is, therefore, when $\Lambda$ does not have a largest element, and part (6c) shows that one may then assume that $\lambda_1<\lambda_2<\lambda_3<\ldots$.

Similar remarks apply to Propositions~\ref{prop:simultaneous_power_factorization_for_bounded sets} and \ref{prop:simultaneous_power_factorization_for_commutative_approximate identity}.

\item A still more precise result is available as Theorem~\ref{thm:positive_simultaneous_power_factorization}; see also Remark~\ref{rem:ordered_theorem_implies_original_theorem}.

\item In the monograph \cite[Theorem~2.9.24]{Dalesbook}, a proof of the pointwise power factorization theorem is given under the assumption that the sequence $\{\alpha_n\}_{n=1}^\infty$ is increasing, but this extra condition is not necessary. Unfortunately, the proof in \cite{Dalesbook} has an error at one point:  on page 313, line~-6, the factorization of $a_{k+1}$ is used to obtain a factorization of $a_{k+1}^{-j}$ as $g(u_{k+1})^{-j} f(u_{k+1})^{j}$, but this is  only sure when the algebra is commutative. The original proof of Allan and Sinclair is a little more complicated, but avoids this problem by using Lemma~\ref{lem:ring_identity}. 
\end{enumerate}
\end{remark}

\begin{proof}[Proof of Theorem~\ref{thm:simultaneous_power_factorization}]
We may assume that  $\delta\leq 1$ and that
\begin{equation}\label{eq:epsilon_assumption}
\epsilon\leq\inf\{\,\delta,(\alpha_n^n-1)\delta: n\geq 1\,\},
\end{equation}
since the right hand side of this inequality is strictly positive due to the properties of $\{\alpha_n\}_{n=1}^\infty$.

We start by choosing a strictly increasing sequence of integers $\seqk{j}$ such that $j_1\geq n_0$ and
\begin{equation}\label{eq:sequence_construction}
\alpha_n\geq 1+\norm{\rba}\Delta^k
\end{equation}
for all $k\geq 1$ and $n\geq j_k$; here $\norm{\rba}$ is the norm of the representation of $\uniba$ on $\bs$. Since $\lim_{n\to\infty}\alpha_n= \infty$, this is possible. We can now apply Propositions~\ref{prop:simultaneous_power_factorization_for_bounded sets} and~\ref{prop:simultaneous_power_factorization_for_commutative_approximate identity} for our given $\epsilon$ and the sequence $\seqk{j}$ as just constructed. This yields sequences $\{b_k\}_{k=0}^\infty$ in $\uniba$ and $\{u_k\}_{k=1}^\infty$ in $\bigcup_{\lambda\in\Lambda}\{\,\lambda\,\}$ with the properties in the parts (1) through (6) of these two propositions. In particular, part (6c) of the present theorem is satisfied.

Using the notation of these two propositions, we let $a=\sum_{i=1}^\infty r(1-r)^{i-1} u_i$, which is as required in part (6a) of the present theorem. This is indeed a well-defined element of $\ba$, since $\sum_{i=1}^\infty\norm{r(1-r)^{i-1} {u_i}}\leq\sum_{i=1}^\infty r(1-r)^{i-1} M=M$.  We also see that $\norm{a}\leq M$, which is part (2) of the present theorem. Note that the parts (4) of the two propositions imply that
\begin{equation}\label{eq:a_equals_b_k_limit}
\lim_{k\to\infty}b_k=a.
\end{equation}

Fix $n\geq 1$. Since $\lim_{k\to\infty}j_k=\infty$, we can choose $k_0\geq 1$ such that $1\leq n\leq j_{k}$ for all $k\geq k_0$. The parts (3) of the two propositions then yield that $\norm{\rba(b_k^{-n})s-\rba(b_{k-1}^{-n})s}<\frac{\epsilon}{2^k}$ for all $k\geq k_0$ and $s\in S$. A telescoping argument subsequently shows that
\begin{equation}\label{eq:Cauchy_sequence}
\norm{\rba(b_m^{-n})s-\rba(b_{l}^{-n})s}<\frac{\epsilon}{2^{l}}
\end{equation}
for all $m\geq l\geq  k_0-1$ and $s\in S$.
This implies that, for all $s\in S$, $\{\rba(b_k^{-n})s\}_{k=1}^\infty$ is a Cauchy sequence in $\bs$. As $\bs$ is now assumed to be a Banach space, we can define
\begin{equation}\label{eq:pointwise_limit}
x_n(s)=\lim_{k\to\infty}\rba(b_k^{-n})s
\end{equation}
for $s\in S$. The `linearity' of $x_n$ in part (5) of the present theorem then follows from~\eqref{eq:pointwise_limit}, and part (6b) of the present theorem now follows from \eqref{eq:pointwise_limit} and the parts (5) of the two propositions.
Since $s=\rba(b_k^n)[\rba(b_k^{-n})s]$ for all $k\geq 1$ and $s\in S$, we see from \eqref{eq:a_equals_b_k_limit} and \eqref{eq:pointwise_limit} that $s=\rba(a^n)x_n(s)$ for all $s\in S$, which is part (1) of the present theorem.

Continuing, we note that it follows from
\eqref{eq:Cauchy_sequence} that
\begin{equation}\label{eq:uniform_limit}
\norm{x_n(s)-\rba(b_{l}^{-n})s}\leq \frac{\epsilon}{2^{l}}
\end{equation}
all $l\geq k_0-1$ and $s\in S$. Since, for all $l\geq k_0-1$, the map $s\mapsto \rba(b_l^{-n})s$ is uniformly continuous on $S$, we conclude from \eqref{eq:uniform_limit} that $x_n$, being a uniform limit of uniformly continuous maps on $S$, is uniformly continuous on $S$. Combining this with the already established relation in part (1), part (3) of the present theorem is now clear.

The parts (1), (2), (3), (5), and (6) of the present theorem have now been established, and we turn to part (4a).

If $1\leq n\leq j_1$, then we can choose $k_0=1$ in the preceding argument, so that \eqref{eq:uniform_limit} holds for $l=k_0-1=0$. Since we also  know from the parts (1) of the two propositions that $b_0=\unit$, we therefore see that 
\begin{equation}\label{eq:x_n_close_to_id_for_first_j_1}
\norm{x_n(s)-s}\leq \epsilon
\end{equation}
for all $n=1,\ldots,j_1$ and $s\in S$. As $n_0\leq j_1$ by the choice of $j_1$, part (4a) of the present theorem has now been established.

We now consider part (4b) of the present theorem, and for this we distinguish two cases.

The first case is where $1\leq n\leq j_1$.  According to \eqref{eq:x_n_close_to_id_for_first_j_1}, we then know that $\norm{x_n(s)}
\leq\epsilon+\norm{s}$ for all $s\in S$. We now distinguish two subcases. If $s\in S$ and $\norm{s}\leq\delta$, then we have, for $n=1,\ldots,j_1$, using \eqref{eq:epsilon_assumption},
\begin{align*}
\norm{x_n(s)}
&\leq\epsilon+\norm{s}\\
&\leq(\alpha_n^n-1)\delta +\delta\\
&=\alpha_n^n\delta\\
&=\alpha_n^n\max(\norm{s},\delta).
\end{align*}
If $s\in S$ and $\norm{s}>\delta$, then we have, for $n=1,\ldots,j_1$, using \eqref{eq:epsilon_assumption} again,
\begin{align*}
\norm{x_n(s)}
&\leq\epsilon+\norm{s}\\
&\leq(\alpha_n^n-1)\delta +\norm{s}\\
&<(\alpha_n^n-1)\norm{s} +\norm{s}\\
&=\alpha_n^n\norm{s}\\
&=\alpha_n^n\max(\norm{s},\delta).
\end{align*}

This establishes part (4b) for the first case, where $n=1,\ldots, j_1$.  We turn to the second case, where  $n>j_1$. Since $\seqk{j}$ was chosen to be strictly increasing, there exists $k^\prime\geq 1$ such that $j_{k^\prime}+1\leq n\leq j_{k^\prime+1}$. Since $n\geq j_{k^\prime}$, we see from \eqref{eq:sequence_construction} that
\begin{equation}\label{eq:lower_bound_for_alpha_n_again}
\alpha_n\geq 1+\norm{\rba}\Delta^{k^\prime}.
\end{equation}
On the other hand, since $n\leq j_{k^\prime+1}$ and $\seqk{j}$ is strictly increasing, we can use our argument above for the choice $k_0=k^\prime+1$. Then \eqref{eq:uniform_limit} is valid for $l=(k^\prime+1)-1=k^\prime$ and all $s\in S$, so that we see that, for all $s\in S$,
\begin{equation}\label{eq:x_n_estimate}
\norm{x_n(s)-\rba(b_{k^\prime}^{-n})s}\leq \frac{\epsilon}{2^{k^\prime}}.
\end{equation}
Using the parts (2) of the two propositions, \eqref{eq:epsilon_assumption},  \eqref{eq:lower_bound_for_alpha_n_again}, and the fact that $\norm{\rba}\geq 1$ because $\rba$ is a unital representation of $\uniba$, we conclude from \eqref{eq:x_n_estimate}  that, for all $s\in S$,
\begin{align*}
\norm{x_n(s)} & \leq \frac{\epsilon}{2^{k^\prime}}+\norm{\rba(b_{k^\prime}^{-n})s}\\
&\leq \frac{\epsilon}{2^{k^\prime}} + \norm{\rba}(\Delta^{k^\prime})^n\norm{s}\\
&\leq \delta + \norm{\rba}\Delta^{k^\prime n}\norm{s}\\
&\leq \max(\norm{s},\delta) + \norm{\rba}\Delta^{k^\prime n}\max(\norm{s},\delta)\\
&\leq \left(1 + \norm{\rba}^n\Delta^{k^\prime n}\right)\max(\norm{s},\delta)\\
&\leq (1+\norm{\rba}\Delta^{k^\prime})^n\max(\norm{s},\delta)\\
&\leq \alpha_n^n\max(\norm{s},\delta).
\end{align*}
The second case, where $n>j_1$, has now been covered, and part (4b) of the present theorem has now been established.

Finally, we consider part (7) of the present theorem. Part (7a) is clear from part (4b). For part (7b), fix $n\geq 1$, and let $\eta>0$ be given. Using \eqref{eq:uniform_limit}, we can choose $l_0\geq 1$ such that $\norm{x_n(s)-\rba(b_{l_0}^{-n})s}\leq\eta/3$ for all $s\in S$. If $\bs=\{\,0\,\}$, then all is clear. If $\bs\neq\{\,0\,\}$, then $\rba(b_{l_0}^{-n})\neq 0$, and there exist $t\geq 1$ and  $s_1,\ldots,s_t\in S$ such that $S\subset\bigcup_{i=1}^t \{\,x\in\bs : \norm{x-s_i}<\eta/(3\norm{\rba(b_{l_0}^{-n})})\,\}$. Let $s\in S$. Then $\norm{s-s_{i_0}}<\eta/(3\norm{\rba(b_{l_0}^{-n})})$ for some $i_0$ such that $1\leq i_0\leq t$, and this implies that 
\begin{align*}
\norm{x_n(s)-x_n(s_{i_0})}&\leq\norm{x_n(s)-\rba(b_{l_0}^{-n})s}\!+\! \norm{\rba(b_{l_0}^{-n})(s-s_{i_0})}\! +\! \norm{\rba(b_{l_0}^{-n})s_{i_0}-x_n(s_{i_0})}\\&<\eta/3 +\eta/3 +\eta/3 \\&= \eta.
\end{align*}  Hence $x_n(S)$ is totally bounded.

For part (7c), fix $n\geq 1$, and let $\eta>0$ be given. We first deal with  the case where $S$ is bounded.
From \eqref{eq:uniform_limit}, we see that there exists $l_0\geq 1$ such that $\norm{\rba}M\norm{x_n(s)-\rba(b_{l_0}^{-n})s}<\eta/3$ for all $s\in S$; this implies that $\norm{x_n(s)-\rba(b_{l_0}^{-n})s}<\eta/3$  for all $s\in S$. Since $S$ is bounded, part (6b) of the present theorem and Lemma~\ref{lem:uniform_convergence_is_preserved_under_action}  imply that there exists $\lambda^\prime\in\Lambda$ such that $\norm{\rba(\e)\rba(b_{l_0}^{-n})s-\rba(b_{l_0}^{-n})s}<\eta/3$ for all $\lambda\geq\lambda^\prime$ and $s\in S$. We then have, for all $\lambda\geq\lambda^\prime$ and $s\in S$,
\begin{align*}
\norm{\rba(\e)x_n(s)-x_n(s)}& \leq \norm{\rba(\e)x_n(s)-\rba(\e)\rba(b_{l_0}^{-n})s} \\
&\quad + \norm{\rba(\e)\rba(b_{l_0}^{-n})s-\rba(b_{l_0}^{-n})s}\\
&\quad + \norm{\rba(b_{l_0}^{-n})s-x_n(s)}\\
& < \norm{\rba} M\norm{x_n(s)-\rba(b_{l_0}^{-n})s}\\
&\quad  + \norm{\rba(\e)\rba(b_{l_0}^{-n})s-\rba(b_{l_0}^{-n})s}\\
&\quad  + \norm{\rba(b_{l_0}^{-n})s-x_n(s)}\\
&<\eta/3 + \eta/3 +\eta/3\\
&=\eta.
\end{align*}
If $\ai$ is commutative, we start again by choosing $l_0\geq 1$ with the property that $\norm{\rba}M\norm{x_n(s)-\rba(b_{l_0}^{-n})s}<\eta/3$ for all $s\in S$, but now we observe that there exists $\lambda^\prime$ such that $\norm{\rba(b_{l_0}^{-n})}\norm{\rba(\e)s-s}<\eta/3$ for all $\lambda\geq\lambda^\prime$. We then have, for all $\lambda\geq\lambda^\prime$ and $s\in S$,
\begin{align*}
\norm{\rba(\e)x_n(s)-x_n(s)}&\leq \norm{\rba(\e)x_n(s)-\rba(\e)\rba(b_{l_0}^{-n})s} \\
&\quad + \norm{\rba(\e)\rba(b_{l_0}^{-n})s-\rba(b_{l_0}^{-n})s}\\
&\quad + \norm{\rba(b_{l_0}^{-n})s-x_n(s)}\\
& \leq \norm{\rba} M\norm{x_n(s)-\rba(b_{l_0}^{-n})s}\\
&\quad + \norm{\rba(b_{l_0}^{-n})\rba(\e)s-\rba(b_{l_0}^{-n})s}\\
&\quad + \norm{\rba(b_{l_0}^{-n})s-x_n(s)}\\
&<\eta/3 +  \norm{\rba(b_{l_0}^{-n})}\norm{\rba(\e)s-s} +\eta/3\\
&<2\eta/3 +\eta/3\\
&=\eta.
\end{align*}

This completes the proof.
\end{proof}

\begin{remark}\label{rem:trick}
If $S$ is bounded and $S\neq\{\,0\,\}$, then one can apply Theorem~\ref{thm:simultaneous_power_factorization} with $\delta=\sup_{s\in S}\norm{s}$ to see that a simultaneous power factorization is possible where
\begin{equation}\label{eq:uniform_bound}
\norm{x_n(s)}\leq \alpha_n^n \sup_{s\in S}\norm{s}
\end{equation}
for all $n\geq 1$ and $s\in S$; this is also possible if $S=\{\,0\,\}$. If $S$ consists of one point, then one obtains the upper bounds that are already in \cite[Theorem~1]{Allan1976}. It can be argued, see \cite[Remark~(i) on p.~37]{Allan1976}, that these upper bounds are then essentially the best possible.

It should be mentioned here that, for bounded $S$, the variation of Theorem~\ref{thm:simultaneous_power_factorization} where (4b) is replaced with \eqref{eq:uniform_bound} can also be obtained rather directly from the power factorization for one point (see \cite[Theorem~1]{Allan1976}) and inspection of the proof thereof.  We shall now describe this. The argument, the idea of which will be generalized in Theorem~\ref{thm:simultaneous_power_factorization_for_maps}, is a slightly improved version of the one given in \cite[p.~115-116]{doran} resp.\ \cite[proof of Corollary~5.2.3.(b)]{Palmer_I}, where a simultaneous non-power factorization for bounded resp.\ compact $S$ is established.

With the assumptions as in Theorem~\ref{thm:simultaneous_power_factorization}, we let $\tilde \bs$ be the space of all bounded uniformly continuous maps $f: S\to \bs$, supplied with the supremum norm. A moment's thought shows that $\tilde\bs$ is a Banach space under pointwise operations, and that there is a natural continuous representation of $\ba$ on $\tilde\bs$ by pointwise action.

Define $\mathrm{id}_S:S\to \bs$ by $\mathrm{id}_S(s)=s$ for all $s\in S$. Then $\mathrm{id}_S\in \tilde\bs$, since $S$ is bounded. Furthermore, for $\lambda\in\Lambda$, we have
\begin{equation}\label{eq:space_of_maps}
\norm{\tilde\rba(\e)\mathrm{id}_S-\mathrm{id}_S}=\sup_{s\in S}\norm{\rba(\e)s-s}.
\end{equation}
Since we have assumed that $\lim_{\lambda\in\Lambda}\rba(\e)s=s$ uniformly on $S$, we see from \eqref{eq:space_of_maps} that $\mathrm{id}_S$ is in the essential subspace for the action of $\ba$ on $\tilde\bs$ via $\tilde\rba$. Therefore, we can apply \cite[Theorem~1]{Allan1976} (and its proof), and this yields the variation of Theorem~\ref{thm:simultaneous_power_factorization} where part (4b) has been replaced with \eqref{eq:uniform_bound}.

The stronger pointwise estimates in part (4b), however, do not seem to be attainable in this fashion, and clearly this whole set-up breaks down if $S$ is not bounded.
\end{remark}

\begin{remark}\label{rem:no_bounded_x_n} If $0\notin\overline{S}$, then one can apply  Theorem~\ref{thm:simultaneous_power_factorization} with $\delta=\inf_{s\in S}\norm{s}>0$. In that case, part (4b) yields that
\[
\norm{x_n(s)}\leq\alpha_n^n\norm{s}
\]
for all $n\geq 1$ and $s\in S$. Hence the maps $x_n:S\to\bs$ are bounded (in the usual operator sense of the word) on $S$ for all $n\geq 1$. It is natural to ask whether this could actually be valid for general $S$. The following example shows that this is not the case, not even for general commutative $\ba$ and bounded $S$.

As in Examples~\ref{ex:first_appearance} and~\ref{ex:second_appearance}, we  consider $\ba=\conto{\mathbb R}$ and the left regular representation of $\ba$, but now with a different $S$. Choose $f_0\in\conto{\mathbb R}$ such that $f_0(t)>0$ for all $t\in\mathbb R$, and let $S^\prime=\{\,f\in\conto{\mathbb R} : 0\leq f(t)\leq f_0(t)\textup{ for all }t\in\mathbb R\,\}$. We know from Example~\ref{ex:first_appearance} that there exists a bounded left approximate identity $\{e_n\}_{n=1}^\infty$ for $\conto{\RR}$ such that $\lim_{n\to\infty}\norm{e_nf-f}=0$ uniformly for $f$ in a subset $S$ of $\conto{\RR}$ containing $S^\prime$, so certainly this is true for $S^\prime$. Hence Theorem~\ref{thm:simultaneous_power_factorization} is applicable.

However, for every $n\geq 1$, there cannot exist $a\in \conto{\mathbb R}$, $C>0$, and a map $x_n:S^\prime\to\conto{\mathbb R}$ such that, for all $f\in S^\prime$, $f=a^n x_n(f)$ and $\norm{x_n(f)}\leq C\norm{f}$. To see this, we argue by contradiction.  Let us assume that these objects exist. First of all, since, in particular, $f_0=a^n x_n(f_0)$, and since $f_0$ has no zeros, we see that $a$ has no zeros. Thus $x_n(f)=a^{-n}f$ for all $f\in S^\prime$. Fix $t_0\in\mathbb R$, and choose a non-zero $f_{t_0}\in S^\prime$ such that $\norm{f_{t_0}}=f_{t_0}(t_0)$; this is possible since $f_0$ is strictly positive in every point. Then
\begin{equation*}
|a^{-n}(t_0)f_{t_0}(t_0)|=|[x_n(f_{t_0})](t_0)|\leq\norm{x_n(f_{t_0})}\leq C\norm{f_{t_0}}=C f_{t_0}(t_0).
\end{equation*}
Since $f_{t_0}(t_0)=\norm{f_{t_0}}\neq 0$, we conclude that $|a^{-n}(t_0)|\leq C$ for all $t_0\in\mathbb R$. This leads to $|a(t_0)|\geq C^{-1/n}$ for all $t_0\in\mathbb R$, contradicting that $a\in\conto{\mathbb R}$.
\end{remark}

\begin{remark}
Part (7) of Theorem~\ref{thm:simultaneous_power_factorization} shows that all sets $x_n(S)$ for $n\geq 1$ inherit crucial properties from $S$. The converse is also true. In fact, one single $n$ (which we can take to be equal to 1) already suffices. More specifically, suppose that $S\subset \bs$ is such that there exist $a\in A$ as in part (6a) of Theorem~\ref{thm:simultaneous_power_factorization}, and a map $x_1:S \to\bs$ such that $s=\rba(a)x_1(s)$ for all $s\in S$ and $\lim_{\lambda}\rba(\e)x_1(s)=x_1(s)$ uniformly on $S$. If $x_1(S)$ is bounded as in part (7a), or totally bounded as in part (7b), then the same holds for $S=\rba(a)(x_1(S))$. In order to show that the uniform convergence on $S$ follows from part (7c), we distinguish two cases. If $x_1(S)$ is bounded, then the fact that  $\norm{\rba(\e)s-s}=\norm{\rba(\e)\rba(a)x_1(s)-\rba(a)x_1(s)}\leq\norm{\rba}\norm{\e a-a}(\sup_{s\in S}\norm{x_1(s)})$ implies evidently that $\lim_{\lambda}\rba(\e)s=s$ uniformly on $S$.
If $\ai$ is commutative, then we note that
\begin{align*}
\norm{\rba(\e)s-s}&=\norm{\rba(\e)\rba(a)x_1(s)-\rba(a) x_1(s)}\\
&=\norm{\rba(a)\rba(\e)x_1(s)-\rba(a) x_1(s)}\\
&\leq \norm{\rba(a)}\norm{\rba(\e)x_1(s)-x_1(s)}
\end{align*}
in order to conclude that $\lim_{\lambda}\rba(\e)s=s$ uniformly on $S$; here part (6a) is used in the second step.

A similar observation, relating the possibility of simultaneous non-power factorization of a set, uniform convergence on the set, and uniform convergence on the factor set, can already be found for uniformly bounded subsets of $\ess{\bs}$ in \cite[Theorem~2.1]{SentillesTaylor}.
\end{remark}

For the sake of completeness, we also include the following result, showing that a simultaneous power factorization can also be valid on subsets of $\bs$ on which there need not be any uniform convergence at all.

\begin{corollary}\label{cor:simultaneous_power_factorization_for_countable_union_of_compact_subsets}
Let $\ba$ be a Banach algebra with a bounded left approximate identity of bound $M\geq 1$, and let $\rba$ be a continuous representation of $\ba$ on the Banach space $\bs$. Suppose that $S=\bigcup_{l=1}^\infty K_l$ is the countable union of non-empty compact subsets $K_l$ of $\ess{\bs}$. Then there exist $a\in \ba$ such that $\norm{a}\leq M$, and, for all $n\geq 1$, a subset $\bs_n$ of $\ess{\bs}$ such that $S=\rba(a^n)\bs_n$.
\end{corollary}

For $n=1$, this result (of which \cite[Corollary~17.6]{doran} for countable subsets of $\ess{\bs}$ is then a special case) follows from the simultaneous non-power factorization for compact subsets of $\ess{\bs}$ via a concrete and simple transformation of the picture. We refer to \cite[proof of Corollary~5.2.3.(a)]{Palmer_I} for details. The same concrete transformation gives the simultaneous power version in Corollary~\ref{cor:simultaneous_power_factorization_for_countable_union_of_compact_subsets} as a result of Theorem~\ref{thm:simultaneous_power_factorization}. It is, therefore, possible to obtain estimates that are valid for the factorization in Corollary~\ref{cor:simultaneous_power_factorization_for_countable_union_of_compact_subsets} from those in Theorem~\ref{thm:simultaneous_power_factorization}. For reasons of space, we refrain from going into this.

\section{Positive simultaneous power factorization}\label{sec:positive_simultaneous_power_factorization}

Let $\oba$ be an ordered Banach algebra that has a positive left approximate identity $\ai$, let $\obs$ be an ordered Banach space, and let $\rba$ be a positive continuous representation of $\oba$ on $\obs$. In this context, it is natural to investigate the existence of a positive factorization. Restricting ourselves to the pointwise non-power case, we have the following question: if $s\in\pos{\ess{\obs}}$, do there always exist $a\in\pos{\oba}$ and $x\in\pos{\obs}$ such that $s=\rba(a)x$?

The answer to the question in this generality is negative. In fact, it can already fail for the left regular representation of $\oba$. It was remarked by Rudin (see \cite{rudin1957}) that there exist positive elements of $\pos{\Ell^1(\RR)}$ that are not the convolution of two elements of $\pos{\Ell^1(\RR)}$: the convolution of two
non-negative integrable functions is always lower semi-continuous, but there exist non-negative integrable functions that are not almost everywhere equal to a lower semi-continuous function. We refer to \cite{Ross} for more (also historical) information concerning this matter and various factorization theorems, with special attention for factorization in abstract harmonic analysis.

In this section, we shall investigate the existence of positive factorizations as they can sometimes be derived from Theorem~\ref{thm:simultaneous_power_factorization}. Looking at Theorem~\ref{thm:simultaneous_power_factorization}, the positivity of $a$ is hardly an issue. If $\ai$ is positive, and if $\pos{\oba}$ is closed, then part (6a) of Theorem~\ref{thm:simultaneous_power_factorization} shows that factorizations produced by Theorem~\ref{thm:simultaneous_power_factorization} will always have positive $a$. The problem lies with the maps $x_n$ for $n\geq 1$. Can we sometimes guarantee that $x_n(S)\subset\pos{\ess{\obs}}$ for $S\subset\pos{\obs}$?

If one is prepared to be content with this property for the first finitely many values of $n$, then part (4a) of Theorem~\ref{thm:simultaneous_power_factorization} gives an obvious sufficient condition. The following result is clear.

\begin{theorem}\label{thm:positive_simultaneous_power_factorization_for_subset_of_interior_positive_cone}

Let $\oba$ be an ordered Banach algebra with a closed positive cone and a positive left approximate identity $\ai$ of bound $M\geq 1$, let $\obs$ be an ordered Banach space, and let $\rba$ be a positive continuous representation of $\ba$ on $\bs$.  Let $S$ be a non-empty subset of $\pos{\obs}$ such that $\lim_\lambda \rba(\e)s=s$ uniformly on $S$, and suppose that $S$ is bounded or that $\ai$ is commutative.

Assume, in addition, that there exists $\eta>0$ such that, for all $s\in S$, $\{\,x\in\ess{\obs} : \norm{x-s}<\eta\,\}\subset\pos{\ess{\obs}}$.

Then, for every superalgebra $B$, $\epsilon$ such that $\epsilon<\eta$, $\delta$, $r$, $n_0$, and sequence $\{\alpha_n\}_{n=1}^\infty$ as in Theorem~\ref{thm:simultaneous_power_factorization}, there exists a simultaneous power factorization as in that theorem for the restricted continuous representation $\ess{\rba}$ of $\oba$ on the Banach space $\ess{\obs}$ with non-empty subset $S$ of $\ess{\obs}$, with the following additional statements:
\begin{enumerate}
\item[(8)] $a\in\pos{\oba}$;
\item[(9)] $x_n(S)\subset\pos{\ess{\obs}}$ for all $n=1,\ldots,n_0$. More precisely:  $\{\,x\in\ess{\obs} : \norm{x-x_n(s)}<\eta-\epsilon\,\}\subset\pos{\ess{\obs}}$ for all $n=1,\ldots,n_0$ and $s\in S$.
\end{enumerate}
\end{theorem}

\begin{remark}\quad
\begin{enumerate}
\item Note that $\pos{\obs}$ need to be closed in Theorem~\ref{thm:positive_simultaneous_power_factorization_for_subset_of_interior_positive_cone}.
\item The result is only non-void if the interior of $\pos{\ess{\obs}}$ is non-empty. It is interesting to note that, in Rudin's (counter)example, the positive cone of the essential subspace of the ordered Banach space in question, i.e.\  $\pos{\Ell^1(\mathbb R)}$, has empty interior. It is unclear to the authors whether this is actually somehow related to the failure of positive factorization for the left regular representation of $\Ell^1(\mathbb R)$.
\item The idea to use estimates as in part (4a) of Theorem~\ref{thm:positive_simultaneous_power_factorization_for_subset_of_interior_positive_cone} to obtain a positive factorization is by no means new. It was already observed by Cohen (see \cite[p.~204]{cohen1959}), using precisely this argument, that a strictly positive continuous function $f$ on a compact group $G$ is a convolution of a strictly positive element $a$ of $\Ell^1(G)$ and a strictly positive continuous function $f_1$. This corresponds to an application of Theorem~\ref{thm:positive_simultaneous_power_factorization_for_subset_of_interior_positive_cone} to  the action of $\Ell^1(G)$ on $\cont{G}$ by convolution; the strict positivity of $a$ follows from part (6a) of Theorem~\ref{thm:simultaneous_power_factorization} once one notes that $\Ell^1(G)$ has a strictly positive bounded left approximate identity (see \cite[p.~203]{cohen1959} for the easy argument). It is now also clear that, for $n_0
\geq 1$, one can, in fact, obtain finitely many factorizations $f=a^{\ast n}\ast f_n$ for $n=1,\ldots,n_0$, where $a\in
\Ell^1(G)$ and the $f_n\in\cont{G}$ are all strictly positive. This can even be achieved simultaneously for all $f$ in a suitable subset (for example, a totally bounded subset) of $\pos{\cont{G}}$ that is bounded below by a strictly positive constant function.
\end{enumerate}
\end{remark}

Part (6b) of Theorem~\ref{thm:simultaneous_power_factorization} gives another sufficient condition for a positive simultaneous power factorization to exists, and this time such that $x_n$ maps $S$ into $\pos{\ess{\obs}}$ for all $n\geq 1$. Once one observes that the $b_k$ in part (6b) are all positive if $\ai$ is positive, the following result is clear.

\begin{theorem}\label{thm:positive_simultaneous_power_factorization}

Let $\oba$ be an ordered Banach algebra with a closed positive cone and a positive left approximate identity $\ai$ of bound $M\geq 1$, let $\obs$ be an ordered Banach space with a closed positive cone, and let $\rba$ be a positive continuous representation of $\ba$ on $\bs$.  Let $S$ be a non-empty subset of $\pos{\bs}$ such that $\lim_\lambda \rba(\e)s=s$ uniformly on $S$, and suppose that $S$ is bounded or that $\ai$ is commutative.

Suppose that there exists a unital ordered Banach superalgebra $\unioba\supset\oba$ such that:
\begin{enumerate}
\item the restricted positive continuous representation $\ess{\rba}$ of $\ba$ on $\ess{\obs}$ extends to a positive continuous unital representation of $\unioba$ on $\ess{\obs}$;
\item $\pos{\unioba}$ is inverse closed in $\unioba$.
\end{enumerate}

With this choice of $\uniba$ in Theorem~\ref{thm:simultaneous_power_factorization}, for every $\epsilon$, $\delta$, $r$, $n_0$, and sequence $\{\alpha_n\}_{n=1}^\infty$ as in that theorem, there exists a simultaneous power factorization as in that theorem for the restricted continuous representation $\ess{\rba}$ of $\oba$ on the Banach space $\ess{\obs}$ with non-empty subset $S$ of $\ess{\obs}$, with the following additional statements:
\begin{enumerate}
\item[(8)] $a\in\pos{\oba}$;
\item[(9)] $x_n(S)\subset\pos{\ess{\obs}}$ for all $n\geq 1$.
\end{enumerate}

\end{theorem}

\begin{remark}\label{rem:ordered_theorem_implies_original_theorem}\quad
\begin{enumerate}
\item 
One might be tempted to think of Theorem~\ref{thm:positive_simultaneous_power_factorization} as a special case of Theorem~\ref{thm:simultaneous_power_factorization} in an ordered context, but actually it is not. It is more precise than the latter result, which it contains as a special case. Indeed, in the context of Theorem~\ref{thm:simultaneous_power_factorization} one can introduce an ordering on $\ba$, $\uniba$, and $\bs$ by taking the spaces themselves as the positive cones. Then positivity of maps, closedness of positive cones and inverse closedness of algebra cones all become a triviality, and Theorem~\ref{thm:positive_simultaneous_power_factorization} is applicable. It then yields all conclusions in Theorem~\ref{thm:simultaneous_power_factorization}, and adds the then trivially true statements in the parts (8) and (9).

\item  In this context, let us include the following small result, with as particular case that $\pos{\uniba}$ is inverse closed if squares in $\uniba$ are positive.
\smallskip

\noindent \textit{Let $\unioba$ be a unital ordered algebra with positive cone $\pos{\unioba}$. If $b^{-2}\in\pos{\unioba}$ for all $b\in\pos{\uniba}\cap\inv{\unioba}$ \ulb in particular: if $b^2\in\pos{\unioba}$ for all $b\in\unioba$\urb, then $\pos{\unioba}$ is inverse closed in $\uniba$.}
\begin{proof}
Suppose that $b\in\pos{\uniba}\cap\inv{\unioba}$. Then $b^{-1}=b^{-1}\unit=(b^{-1})^2b\geq 0$.
\end{proof}
\end{enumerate}

\end{remark}

The next desirable step would be to exhibit a class of ordered Banach algebras $\ba$ such that an ordered superalgebra $\unioba$ as in Theorem~\ref{thm:positive_simultaneous_power_factorization} exists for all (or at least for a reasonably large class of) positive representations of $\ba$ on ordered Banach spaces $\obs$. Rudin's example shows, however, that positive factorization already fails for the left regular representation of $\Ell^1(\mathbb R)$. Since  $\Ell^1(\mathbb R)$ is a commutative Banach lattice algebra, and since the left regular representation is even an isometric lattice homomorphism of $\Ell^1(\RR)$ into the regular operators on $\Ell^1(\RR)$ (this is true for the left regular representation of $\Ell^1(G)$ for an arbitrary locally compact group; see \cite[Proposition~3.3]{Arendt}), the situation here is about as nice as one can get, apart, perhaps, from the positive cone of the representation space having empty interior. Possibly one should have modest expectations about such general theorems. Theorem~\ref{thm:positive_simultaneous_power_factorization_for_banach_algebras_of_functions} is a result in this direction, but the matter as a whole is still unclear and more research seems desirable.

A natural candidate for an ordered Banach superalgebra of $\oba$ is its unitization in its natural ordering. Certainly, a positive representation of $\oba$ extends to a positive representation of its unitization, but, as the next result shows, the positive cone of the unitization will hardly ever be inverse closed for the ordered Banach algebras that one is most likely to encounter in practice.  The unitization of every non-zero Banach lattice algebra, for example, does not have this property. We recall that the positive cone of an ordered normed space $\obs$ is said to be normal if there exists $\alpha\geq 0$ such that $\norm{x}\leq\alpha\norm{y}$ whenever $x,y\in\obs$ are such that $0\leq x\leq y$.

\begin{proposition}\label{prop:positive_cone_of_oba_1_is_not_inverse_closed_if_oba_is_normal}Let $\oba$ be a ordered Banach algebra with positive cone $\pos{\oba}$, and let $\unioba$ be its unitization with positive cone $\pos{\unioba}={ \mathbb R}_{\geq 0}\oplus\pos{\oba}$. Assume that $\pos{\oba}$ is proper and closed, or that $\pos{\oba}$ is normal. Then $\pos{\uniba}$ is inverse closed in $\unioba$ if and only if $\pos{\oba}=\{\,0\,\}$.
\end{proposition}

\begin{proof}
If $\pos{\oba}=\{\,0\,\}$, then clearly $\pos{\unioba}=\mathbb R_{\geq 0}$ is inverse closed in $\uniba$. 

We first establish the converse for the case where $\pos{\oba}$ is proper and closed.
 Let $a\in\pos{\oba}$. Then $ta\in\pos{\oba}$ for all $t\geq 0$, and $\norm{ta}<1$ for all sufficiently small $t\geq 0$. Hence, for all sufficiently small $t\geq 0$, $1+ta\in\pos{\unioba}$ is invertible in $\unioba$ with inverse $(1+ta)^{-1}=1+\sum_{n=1}^\infty(-ta)^n$. Since $(1+ta)^{-1}\in\pos{\unioba}$ by assumption, we see that
$\sum_{n=1}^\infty(-ta)^n\geq 0$ for all sufficiently small $t\geq 0$. This implies that $\sum_{n=1}^\infty  (-1)^n t^{n-1} a^n\geq 0$ for all sufficiently small $t>0$. Letting $t\downarrow 0$, and using that $\pos{\oba}$ is closed, we conclude that $-a\geq 0$. Since $\pos{\oba}$ is assumed to be proper, we see that $a=0$.

We now establish the converse for the case where $\pos{\oba}$ is normal. Let $a\in\pos{\oba}$. As in the previous case, this implies that $\sum_{n=1}^\infty(-ta)^n\geq 0$ for all sufficiently small $t\geq 0$. Hence $0\leq ta\leq \sum_{n=2}(-ta)^n$ for all sufficiently small $t\geq 0$. Since $\pos{\oba}$ is normal, we know that there exists $\alpha\geq 0$ such that
\[
\norm{ta}\leq \alpha\norm{\sum_{n=2}^\infty(-ta)^n}\leq\alpha\sum_{n=2}^\infty (t\norm{a})^n=\alpha \frac{t^2\norm{a}^2}{1-t\norm{a}}
\]
for all sufficiently small $t\geq 0$. Hence
\[
\norm{a}\leq \alpha t\frac{\norm{a}^2}{1-t\norm{a}}
\]
for all sufficiently small $t>0$, which implies that $a=0$.

\end{proof}

\begin{remark}\label{rem:centralizer_algebras} It is worthwhile to mention that there can be other natural candidates for unital Banach superalgebras in Theorem~\ref{thm:positive_simultaneous_power_factorization} to work with: the left centralizer algebra and the double centralizer algebra of $\oba$. According to \cite[Theorem~4.1]{DJW}, a continuous non-degenerate representation $\rba$ of a normed algebra $\ba$ with a bounded approximate left identity $\ai$ on a Banach space $\bs$ gives rise to a continuous unital representation $\overline{\pi}$ of the left centralizer algebra $\leftcent{\ba}$ of $\ba$ on $\bs$ that is compatible with the canonical homomorphism $\leftreg: \ba\to\leftcent{\ba}$ as provided by the left regular representation of $\ba$, i.e.\ is such that $\rba=\overline{\rba}\circ\leftreg$.  It is given by $\overline{\pi}(L)=\textup{SOT}-\lim_{\lambda}\rba(L(\e))$ for $L\in\leftcent{\ba}$. If $\oba$ is an ordered Banach algebra, if $\ai$ is positive, if $\obs$ is a Banach space with a closed positive cone, and if $\rba$ is positive, then $\leftcent{\oba}$ is an ordered Banach algebra and $\overline{\pi}$ is positive. One can now apply Theorem~\ref{thm:simultaneous_power_factorization} to this context, where $\ba$ is to be replaced with the closure $\overline{\leftreg(\oba)}$ of $\leftreg(\oba)$ in $\leftcent{\oba}$, $\ai$ is to be replaced with $\{\leftreg(\e)\}_{\lambda\in\Lambda}$, $\unioba$ is chosen to be $\leftcent{\oba}$, and $\rba$ is to be replaced with $\overline{\rba}$. Theorem~\ref{thm:simultaneous_power_factorization} will then produce a simultaneous power factorization of the form
\begin{equation}\label{eq:centralizer_in_factorization}
s=\overline{\rba}(L^n)x_n(s)
\end{equation}
for some $L\in\overline{\leftreg(\oba)}$. If $\lambda(A)$ is closed in  $\leftcent{\oba}$ (e.g.\ if $\oba$ also has a bounded right approximate identity), then $L=\leftreg(a)$ for some $a\in\ba$, and the factorization takes it usual form.

The point is that  $\leftcent{\oba}$ can have better properties than the unitization of $\ba$. More specifically: it can be the case that $\pos {\leftcent{\oba}}$ is inverse closed, whereas\textemdash see Proposition~\ref{prop:positive_cone_of_oba_1_is_not_inverse_closed_if_oba_is_normal}\textemdash the positive cone of the unitization of $\oba$ quite often is not. In that case, Theorem~\ref{thm:positive_simultaneous_power_factorization} will produce a factorization as in \eqref{eq:centralizer_in_factorization} with $L\in\pos{\leftcent{\oba}}$. If $\leftreg$ is a bipositive topological embedding of $\oba$ into  $\leftcent{\oba}$ (e.g.\ if $\oba$ has a closed positive cone and if $\oba$ also has a positive bounded right approximate identity), then $L=\leftreg(a)$ for some $a\in\pos{\oba}$, and a positive simultaneous power factorization has been obtained.

The situation in the previous paragraph can actually occur. If $\oba=\conto{\ts}$ for a locally compact Hausdorff space $\ts$, then we know from Proposition~\ref{prop:positive_cone_of_oba_1_is_not_inverse_closed_if_oba_is_normal} that the positive cone of its unitization is not inverse closed. However, this \emph{is} quite obviously the case for its centralizer algebra $\contb{\ts}$. Hence we still have a positive simultaneous power factorization theorem for $\conto{\ts}$.

Similar remarks apply to the double centralizer algebra of an ordered Banach algebra that has a closed positive cone, a positive bounded left approximate identity, and a positive bounded right approximate identity. In that case, \cite[Theorem~4.5]{DJW} can be used.
\end{remark}

The preceding remark motivates the choice of the superalgebra $\uniba$ in the final result of this section. Strictly speaking, it has already been established in that remark, since, by the Gelfand-Naimark theorem, the algebra $\oba$ in Theorem~\ref{thm:positive_simultaneous_power_factorization_for_banach_algebras_of_functions} below is isometrically and bipositively isomorphic to an algebra $\conto{\ts}$ for some locally compact Hausdorff space $\ts$. As the proof below shows, one can also avoid invoking this representation theorem, and simply apply the observation that positive cones of unital algebras  of functions are obviously inverse closed.

\begin{theorem}\label{thm:positive_simultaneous_power_factorization_for_banach_algebras_of_functions}
Let $\Omega$ be a non-empty set, and let $\oba$ be an ordered Banach algebra of bounded functions on $\Omega$, supplied with the supremum norm. Then $\oba$ has a positive 1-bounded approximate identity.

Let $\obs$ be an ordered Banach space with a closed positive cone, and let $\rba$ be a positive continuous representation of $\ba$ on $\bs$.  Let $S$ be a non-empty subset of $\pos{\obs}$ such that $\lim_\lambda \rba(\e)s=s$ uniformly on $S$.

Let
\[
\unioba=\{\,g:\Omega\to\FF : g\textup{ is bounded and }g f\in\oba\textup{ for all } f\in\oba\,\}.
\]
be the normalizer of $\oba$ in the bounded functions on $\Omega$, supplied with the supremum norm. Then the unital ordered superalgebra $\uniba$ of $\oba$ satisfies the hypotheses under \ulb 1\urb\ and \ulb 2\urb\ in Theorem~\ref{thm:positive_simultaneous_power_factorization}.

Therefore, with this choice of $\uniba$ in Theorem~\ref{thm:simultaneous_power_factorization}, for every $\epsilon$, $\delta$, $r$, $n_0$, and sequence $\{\alpha_n\}_{n=1}^\infty$  as in that theorem, there exists a simultaneous power factorization as in that theorem for the restricted continuous representation $\ess{\rba}$ of $\oba$ on the Banach space $\ess{\obs}$ with non-empty subset $S$ of $\ess{\obs}$, with the following additional statements:
\begin{enumerate}
\item[(8)] $a\in\pos{\oba}$;
\item[(9)] $x_n(S)\subset\pos{\ess{\obs}}$ for all $n\geq 1$;
\item[(10)] one can take $M=1$ in part \ulb 2\urb\ of Theorem~\ref{thm:simultaneous_power_factorization}.
\end{enumerate}
\end{theorem}

\begin{proof}
If $\FF=\CC$, then $\oba$ is a complex $\textup{C}^*$-algebra, so that it has a 1-bounded positive approximate identity; see e.g.\ \cite[Theorem~3.1.1]{Murphy}.
If $\FF=\RR$, then we consider the algebra of complex functions $\oba_\CC=\oba\oplus\mathrm{i}\oba$, supplied with the supremum norm.  This is a complex $\textup{C}^*$-algebra, and a 1-bounded positive approximate identity for $\oba_\CC$ is contained in $\oba$. We conclude that, in both cases, $\oba$ has a 1-bounded positive approximate identity.

It is clear that  $\unioba$ is a unital ordered Banach superalgebra of $\oba$.  Since it contains $\oba$ as a left ideal, and since $\oba$ contains a positive left approximate identity for itself, we conclude from  \cite[Theorem~3.1]{DJW} that the non-degenerate positive continuous representation of $\oba$ on $\ess{\obs}$ extends (uniquely) to a positive continuous unital representation of $\uniba$ on $\ess{\obs}$. Since we are working with functions, $\pos{\unioba}$ is trivially inverse closed in $\uniba$. Hence the hypotheses in Theorem~\ref{thm:positive_simultaneous_power_factorization} are satisfied, and an application of this result completes the present proof.
\end{proof}

\section{Simultaneous power factorization for sets of maps}\label{sec:simultanous_power_factorization_for_sets_of_maps}

According to \cite[p.~251]{doran},  Collins and Summer (see \cite{CollinsSummers}) and Rieffel (see \cite[proof of Lemma~1]{Rieffel1969}) were the first to realize that it can sometimes be fruitful to introduce an auxiliary Banach module to solve a problem at hand. For example, if one wants to prove that a convergent sequence in a Banach module can be factored termwise, then this can be done by considering the Banach space of all convergent sequences in the pertinent Banach space. This is a Banach module over the same algebra in a natural fashion, and an application of a factorization theorem in that context to the point corresponding to the original sequence will give what one wants.  The argument in Remark~\ref{rem:trick} (and in the references given therein) is another application of this idea of introducing an auxiliary module.

We shall now apply this idea to Theorem~\ref{thm:simultaneous_power_factorization}, which allows us to obtain simultaneous pointwise power factorization theorems for sets of maps. The most general set-up, formulated with an otherwise unspecified Banach space $\bs^\prime$, seems to be the following.

\begin{theorem}\label{thm:simultaneous_power_factorization_for_maps}
Let $\ba$ be a Banach algebra that has a bounded left approximate identity $\ai$ of bound $M\geq 1$, let $\bs$ be a Banach space, and let $\rba$ be a continuous representation of $\ba$ on $\bs$.  Let $S$ be a non-empty subset of $\bs$ such that $\lim_\lambda \rba(\e)s=s$ uniformly on $S$, and suppose that $S$ is bounded or that $\ai$ is commutative.

Let $\Omega$ be a non-empty set, and let $\bs^\prime$ be a Banach space of bounded maps from $\Omega$ into $\bs$, supplied with the supremum norm. Suppose that $X^\prime$ is invariant under the natural pointwise action of $\ba$ on $\bs$-valued maps on $\Omega$, so that there is a natural continuous representation $\rba^\prime$ of $\ba$ on $\bs^\prime$.

Choose a unital superalgebra $\uniba$ of $\ba$ such that $\rba^{\prime}$ extends to a continuous unital representation, again denoted by $\rba^\prime$, of $\uniba$ on $\bs^\prime$.

Let $S^\prime$ be the set of all $f\in\bs^\prime$ such that $f(\Omega)\subset S$.

With this choice of $\uniba$ in Theorem~\ref{thm:simultaneous_power_factorization}, for every $\epsilon$, $\delta$, $r$, $n_0$, and sequence $\{\alpha_n\}_{n=1}^\infty$ as in that theorem, there exists a simultaneous power factorization as in that theorem for the continuous representation  $\rba^\prime$ of $\ba$ on the Banach space $\bs^\prime$ with non-empty subset $S^\prime$ of $\bs^\prime$.
\end{theorem}

\begin{proof}
If $S$ is bounded, then so is $S^\prime$. Furthermore, if $f\in S^\prime$ and $\lambda\in\Lambda$, then $\norm{\rba^\prime(\e)f-f}=\sup_{\omega\in\Omega}\norm{\rba(\e)[f(\omega)]-f(\omega)}\leq \sup_{s\in S}\norm{\rba(\e)s-s}$. We conclude that $\lim_{\lambda}\norm{\rba^\prime(\e)f-f}=0$ uniformly on $S^\prime$. Therefore, Theorem~\ref{thm:simultaneous_power_factorization} can be applied.
\end{proof}

Naturally, one can always choose $\uniba$ to be the unitization of $\ba$. 

For reasons of space, we refrain from translating all statements in Theorem~\ref{thm:simultaneous_power_factorization} into the context of Theorem~\ref{thm:simultaneous_power_factorization_for_maps}. Let us note, however, that one of the consequences  is that there exist $a\in\ba$, and, for all $n\geq 1$, a map $x_n:S^\prime\to\bs^\prime$ such that $f(\omega)=\rba(a^n)\left([x_n(f)](\omega)\right)$ for all $f\in S^\prime$ and $\omega\in\Omega$.

Theorem~\ref{thm:simultaneous_power_factorization_for_maps} has several special cases. For general $\Omega$, one can let $X^\prime$ be the space of all bounded maps $f:\Omega\to\bs$. If $\Omega$ is a topological space, one can consider all $f:\Omega\to\bs$ that are bounded and continuous.  If $\Omega$ is a metric space, one can consider all $f:\Omega\to\bs$ that are bounded and uniformly continuous; this was done in Remark~\ref{rem:trick} for $\Omega=S$. Variations involving the vanishing of $f$ at a subset of $\Omega$ and / or at infinity can also be incorporated. If $\Omega$ has a differentiable structure, versions for sets of bounded $\bs$-valued maps possessing a certain degree of smoothness can conceivably be established.

Another class of special cases of Theorem~\ref{thm:simultaneous_power_factorization_for_maps} occurs when we lay emphasis on an ordering that $\Omega$ can have, rather than a possible topology.  In the spirit of results that are concerned with non-power factorization of one convergent sequence (see e.g.\ \cite[Corollary~7.11]{BonsallDuncanbook}, \cite[Theorems~17.4 and~17.5]{doran}, \cite[Corollary~5.2.3.c and Corollary~5.2.4]{Palmer_I}, and \cite[proof of Lemma~1]{Rieffel1969}), we have the following.

\begin{theorem}\label{thm:simultaneous_power_factorization_for_convergent_nets}
Let $\ba$ be a Banach algebra that has a bounded left approximate identity $\ai$ of bound $M\geq 1$, let $\bs$ be a Banach space, and let $\rba$ be a continuous representation of $\ba$ on $\bs$.  Let $S$ be a non-empty subset of $\bs$ such that $\lim_\lambda \rba(\e)s=s$ uniformly on $S$, and suppose that $S$ is bounded or that $\ai$ is commutative.

Let $\Omega$ be a directed set, and let $\bs^\prime$ be the  Banach space of all bounded convergent nets $f:\Omega\to \bs$, supplied with the supremum norm, or, alternatively, let $\bs^\prime$ be the Banach space of all  bounded nets $f:\Omega\to \bs$ that converge to zero, supplied with the supremum norm; in the latter case, we assume that $0\in\overline S$.  Let $\rba^\prime$ be the natural continuous representation of $\ba$ on $\bs^\prime$ by pointwise operations, and let $S^\prime$ be the set of all elements $\{f_\omega\}_{\omega\in\Omega}$ of $X^\prime$ such that $f_\omega\in S$ for all $\omega\in\Omega$.

Choose a unital superalgebra $\uniba$ of $\ba$ such that $\rba^{\prime}$ extends to a continuous unital representation, again denoted by $\rba^\prime$, of $\uniba$ on $\bs^\prime$.

With this choice of $\uniba$ in Theorem~\ref{thm:simultaneous_power_factorization}, for every $\epsilon$, $\delta$, $r$, $n_0$, and sequence $\{\alpha_n\}_{n=1}^\infty$ as in that theorem, there exists a simultaneous power factorization as in that theorem for the continuous representation  $\rba^\prime$ of $\ba$ on the Banach space $\bs^\prime$ with non-empty subset $S^\prime$ of $\bs^\prime$.

\end{theorem}

As above, one can always choose $\uniba$ to be the unitization of $\ba$.

Again, we refrain from translating all statements in Theorem~\ref{thm:simultaneous_power_factorization} to the context of Theorem~\ref{thm:simultaneous_power_factorization_for_convergent_nets}. One of the consequences is that there exist $a\in\ba$, and, for all $n\geq 1$, a map $x_n:S^\prime\to\bs^\prime$ such that $f_\omega=\rba(a^n)x_n(f)_\omega$ for all $\{f_\omega\}_{\omega\in\Omega}\in S^\prime$ and $\omega\in\Omega$. The point is, of course, that all nets $\{x_n(f)_\omega\}_{\omega\in\Omega}$ for $f\in S^\prime$ are automatically bounded and convergent (or bounded and convergent to zero) again, since they are elements of $\bs^\prime$, 

As a particular case, using Lemma~\ref{lem:totally_bounded_subsets}, we see that there exists a simultaneous pointwise power factorization for all convergent nets $\{f_\omega\}_{\omega\in\Omega}$ in a totally bounded subset $S$ of $\ess{\bs}$. A special case of this, in turn, occurs when $\bs=\ess{\bs}$ and a convergent sequence $\{s_l\}_{l=1}^\infty$ in $\bs$ is given.  One can then take $\Omega=\{\,1,2,\ldots\,\}$ and $S=\{\,s_l : l\geq 1\,\}$. Since $S$ is a totally bounded subset of $\bs$, we see, specializing still further to $n=1$, that there exists $a\in\ba$ such that $s_l=\rba(a)s^\prime_l$ for all $l\geq 1$, and where the convergent sequence $\{s^\prime_l\}_{l=1}^\infty$ in $\bs$ converges to zero if  $\{s_l\}_{l=1}^\infty$ does. Thus the termwise non-power factorization results for sequences in the references prior to Theorem~\ref{thm:simultaneous_power_factorization_for_convergent_nets} are specializations of the theorem.

\medskip

It is obvious how a similar device can be applied to Theorems~\ref{thm:positive_simultaneous_power_factorization_for_subset_of_interior_positive_cone},~ \ref{thm:positive_simultaneous_power_factorization}, and~\ref{thm:positive_simultaneous_power_factorization_for_banach_algebras_of_functions}.  Under the appropriate hypotheses, to be found in these theorems, ordered versions of Theorems~\ref{thm:simultaneous_power_factorization_for_maps} and~\ref{thm:simultaneous_power_factorization_for_convergent_nets} can be established without any further actual proof being necessary. The results thus obtained assert the existence of a positive simultaneous pointwise power factorization (with various extra properties originating from Theorem~\ref{thm:simultaneous_power_factorization}) for sets of bounded maps (including sets of bounded convergent nets, and sets of bounded nets converging to zero) with values in a subset $S$ of the positive cone of an ordered Banach space, where $S$ is such that $\lim_{\lambda}\norm{\e s-s}=0$ uniformly on $S$ for some positive bounded left approximate identity $\ai$ of $\ba$, and where $S$ is bounded or $\ai$ is commutative.  For reasons of space, we refrain from formulating the six ensuing results.

\section{Worked example}\label{sec:worked_example}

In this section, we shall show how the Ansatz in part (6) of Theorem~\ref{thm:simultaneous_power_factorization} can be used in a concrete case to find an explicit positive simultaneous power factorization with all properties as in Theorem~\ref{thm:simultaneous_power_factorization}. One could say that, in this case, Theorem~\ref{thm:simultaneous_power_factorization} is strictly speaking not needed, since the existence of the factorization follows `by inspection'. In practice, however, one might not start any investigations into this direction at all, without the theoretical reassurance that the factorization is actually possible. At first sight, there seem to be (and there are) various technical difficulties to overcome when one wants to find an explicit factorization, and it is not immediately obvious how to do this. If one did not know beforehand that success of the search is guaranteed, one might even consider such success unlikely.

Our example concerns the context of Examples~\ref{ex:first_appearance} and~\ref{ex:second_appearance}, and Remark~\ref{rem:no_bounded_x_n}, where the (real or complex) ordered Banach algebra $\oba$ is $\conto{\mathbb R}$, and where $\rba$ is the positive continuous left regular representation of $\ba$. As unital ordered Banach superalgebra of $\ba$ we choose $B=\contb{\RR}$; its  identity element is the constant function $\onefunction$. Clearly, $\rba$ extends to a positive continuous unital representations of $\contb{\RR}$ on $\conto{\RR}$, defined by pointwise multiplication again. We shall omit the symbol $\rba$ from now on.

As in Example~\ref{ex:first_appearance}, we choose, for every integer $\nu\geq 1$, a function $e_\nu\in\conto{\RR}$ that takes values in $[0,1]$, equals 1 on $[-\nu,\nu]$, and equals 0 on $(-\infty,-\nu-1]\cup[\nu+1,\infty)$. Then $\{e_\nu\}_{\nu=1}^\infty$ is  a positive 1-bounded approximate identity for $\ba$ that is clearly commutative.

As in Example~\ref{ex:first_appearance}, we choose $f_0\in\conto{\mathbb R}$ such that $f_0(t)\geq 0$ for all $t\in\RR$ and such that $\norm{f_0}>1$, and we let
\[
S=\{\,f\in\pos{\conto{\mathbb R}} : f(t)\leq f_0(t)\textup{ for all }t\in\mathbb R\textup{ such that }f_0(t)\leq 1\,\}.
\]
As already noted in Example~\ref{ex:first_appearance},  $S$ is non-empty and unbounded, because it contains functions of arbitrarily large norm that have compact support in the non-empty open set $\{t\in\RR : f_0(t)>1\,\}$.  It was shown in Example~\ref{ex:first_appearance} that $\lim_{\nu\to\infty}\norm{e_\nu f-f}=0$ uniformly for $f\in S$. Clearly, $S\subset\pos{\conto{\RR}}$.

After these preliminary remarks and recollections, we see that Theorem~\ref{thm:positive_simultaneous_power_factorization_for_banach_algebras_of_functions} applies in this context. Hence a positive simultaneous power factorization for $S$ exists, with all additional properties as in Theorem~\ref{thm:simultaneous_power_factorization}. Even for this simple example, this is a non-trivial statement. Disregarding everything else in Theorem~\ref{thm:simultaneous_power_factorization}, it is, in fact, not even immediately clear that a power factorization for $f_0$ alone exists, even though we know, of course, already much longer from \cite[Theorem~1]{Allan1976} that this \emph{is} possible. Indeed, if $f_0$ has no zeros, and if $a\in\conto{\RR}$ and $x_n(f_0)\in\conto{\RR}$ for $n\geq 1$ are such that $f_0=a^n x_n(f_0)$, then $a$ cannot have any zeros either, and we must have that $x_n(f_0)=a^{-n}f_0$ for all $n\geq 1$. However, since $a$ vanishes at infinity, $a^{-1}(t)$ diverges to infinity as $|t|\to\infty$. Since the rate of this divergence increases with $n$, it is not entirely obvious how one can guarantee that $a^{-n}f_0$ still vanishes at infinity for arbitrarily large $n$.

Fortunately, as noted in Remark~\ref{rem:Ansatz}, we know that $a$, and, in fact, the whole positive simultaneous power factorization, can be constructed using only the sequence $\{e_\nu\}_{\nu=1}^\infty$.
More precisely, if we fix $r$ such that $0<r<1/(1+1)=1/2$, then, according to the parts (6a) and (6c) of Theorem~\ref{thm:simultaneous_power_factorization}, there exists a simultaneous power factorization as in Theorem~\ref{thm:simultaneous_power_factorization}, where $a$ is of the form
\begin{equation}\label{eq:Ansatz_for_a}
a=\sum_{i=1}^\infty r(1-r)^{i-1}e_{\nu_i}
\end{equation}
for a strictly increasing sequence $\{\nu_i\}_{i=1}^\infty$. Furthermore, the maps $x_n$ for $n\geq 1$ can be chosen to be as in part (6b) of Theorem~\ref{thm:simultaneous_power_factorization}.

All in all, we merely need to find a suitable strictly increasing sequence $\{\nu_i\}_{i=1}^\infty$, and we shall now embark on doing so. Our approach is to work with the Ansatz for $a$ as in \eqref{eq:Ansatz_for_a}, and then go through the assertions in Theorem~\ref{thm:simultaneous_power_factorization} one by one, each time requiring that it be satisfied with our choice of $\{\nu_i\}_{i=1}^\infty$. As we shall see, this will result in three conditions, all three of the form that each $\nu_i$ be larger than $N(i)$, where $\{N(i)\}_{i=1}^\infty$ is a strictly increasing sequence of strictly positive integers. If these three conditions are all met, then the simultaneous power factorization exists with the corresponding $a$ as in \eqref{eq:Ansatz_for_a}, and with all properties as in Theorem~\ref{thm:simultaneous_power_factorization}. Since it is clearly possible to meet these three sufficient conditions simultaneously, this will give a demonstration `by hand' of the existence of a positive simultaneous power factorization as in Theorem~\ref{thm:positive_simultaneous_power_factorization_for_banach_algebras_of_functions}. Moreover, the lower bounds $N(i)$ will be defined explicitly in terms of the given function $f_0$. In principle, this enables one to determine a completely explicit positive simultaneous power factorization for any concretely given $f_0$.

We start by fixing $r$. Theorem~\ref{thm:simultaneous_power_factorization} guarantees success if $0<r<1/2$, but as long as the statement in part (6b) of Theorem~\ref{thm:simultaneous_power_factorization} on the $b_k^{-1}$ being in a certain Banach subalgebra of $\contb{\RR}$ is not required to hold, the whole construction will actually work if $0<r<1$. We shall therefore fix $0<r<1$ for the time being, and assume that $0<r<1/2$ only when this particular statement in part (6b) is considered at the end of this example.

We need some preparations.

The graph of the strictly positive element $a$ of $\conto{\RR}$ as in \eqref{eq:Ansatz_for_a} resembles a two-dimensional step pyramid that is infinitely wide and that has countably many eroded steps at height $(1-r)^i$ for $i\geq 0$. More precisely, it follows from an easy pointwise calculation that
\begin{align}
a(t)&=1&&\textup{ if }|t|\leq \nu_1,\label{eq:a_value_1}\\
a(t)&=(1-r)^{i-1}(1-r+re_{\nu_i}(t))&&\textup{ if }i\geq 1 \textup{ and }|t|\in[\nu_{i},\nu_{i}+1],\label{eq:a_value_transition}\\
\intertext{and}
a(t)&=(1-r)^i&&\textup{ if }i\geq 1 \textup{ and }|t|\in[\nu_i+1,\nu_{i+1}]\label{eq:a_value_interval}.
\intertext{Since $(1-r)^{i-1}(1-r+re_{\nu_i}(t))\geq (1-r)^{i-1}(1-r)=(1-r)^i$ for all $t\in\RR$, we have our basic equality}
a^{-1}(t)&=1&&\textup{ if }|t|\leq \nu_1,\label{eq:a_inverse_equals_one}
\intertext{and basic inequalities}
a^{-1}(t)&\leq (1-r)^{-i}&&\textup{ if }i\geq 1 \textup{ and }|t|\in[\nu_{i},\nu_{i+1}].\label{eq:a_inverse_upper_bound}
\end{align}
When multiplying with $a^{-n}(t)$ for a fixed $n\geq 1$, the troublesome blow-up factor $(1-r)^{-in}$ for $|t|\in[\nu_i+1,\nu_{i+1}]$ becomes progressively worse as $i\to\infty$. Fortunately, it does so at a controlled rate, namely, exponentially in $i$. As we shall see, this allows us to remedy these blow-ups by letting the sequence $\nu_i$ tend to infinity quickly enough (relative to the decay of $f_0$), forcing that these exponential growths for $n=1,2,\ldots$ are all countered by one super-exponential decay in the relevant estimates.

Let us define
\[
x_n(f)=a^{-n}f
\]
for $n\geq 1$ and $f\in S$.

The first thing to be taken care of is to ensure that $x_n(f)\in\conto{\RR}$ for all $n\geq 1$ and $f\in S$. For this, we select a strictly increasing sequence $\{N_1(i)\}_{i=1}^\infty$ of integers $N_1(i)\geq 1$ such that $0\leq f_0(t)\leq e^{-i^2}$ for all $t$ such that $|t|\geq N_1(i)$. We shall assume in the remainder of this example that $\{\nu_i\}_{k=1}^\infty$ is such that $\nu_i\geq N_1(i)$ for all $i\geq 1$.  We claim that then $a^{-n}f\in\conto{\RR}$ for all $n
\geq 1$ and $f\in S$.  To see this, we fix $n\geq 1$ and $f\in S$. If  $|t|\geq \nu_1$, there exists $i\geq 1$ such that $|t|\in [\nu_i,\nu_{i+1}]$. Then $|t|\geq\nu_i\geq N_1(i)$, so that $0\leq f_0(t)\leq e^{-i^2}<1$. In that case, we also know that $0\leq f(t)\leq f_0(t)$. Hence, for such $t$, we see from \eqref{eq:a_inverse_upper_bound} that
\begin{align*}
|[x_n(f)](t)|&=|a^{-n}(t)f(t)|\\
&\leq (1-r)^{-in}f_0(t)\\
&\leq (1-r)^{-in} e^{-i^2}.
\end{align*}
Since
\begin{equation}\label{eq:gaussian_limit_is_zero}
\lim_{i\to\infty}(1-r)^{-in} e^{-i^2}=0,
\end{equation}
this implies that $x_n(f)\in\conto{\RR}$, as desired.

It is clear that $a\in\pos{\conto{\RR}}$ and that $x_n(f)\in\pos{\conto{\RR}}$ for all $n\geq 1$ and $f\in S$.

We shall now start investigating which further conditions on the sequence $\{\nu_i\}_{i=1}^\infty$ are sufficient for our construction to satisfy the parts (1) through (6) of Theorem~\ref{thm:simultaneous_power_factorization}.

Part (1) of Theorem~\ref{thm:simultaneous_power_factorization} is obviously satisfied, so that we do have a simultaneous power factorization. As already observed, this factorization is clearly positive.

Part (2) of Theorem~\ref{thm:simultaneous_power_factorization} is true for all choices of $\{\nu_i\}_{i=1}^\infty$.

Turning to part (3) of Theorem~\ref{thm:simultaneous_power_factorization}, we claim that the maps $x_n: S\to\conto{\RR}$ are uniformly continuous on $S$ for all $n\geq 1$. To see this, fix $n\geq 1$, and let $\eta>0$ be given. Fix $i_0\geq 1$ such that $2(1-r)^{-in}e^{-i^2}<\eta$ for all $i\geq i_0$.  Now consider $f_1,f_2\in S$. If $|t|\leq\nu_1$, then
\begin{align}\label{eq:uniform_continuity_one}
\begin{split}
|[x_n(f_1)-x_n(f_2)](t)|&=|a^{-n}(t)[f_1(t)-f_2(t)]|\\
&=|f_1(t)-f_2(t)|\\
&\leq\norm{f_1-f_2}.
\end{split}
\end{align}
If $|t|\geq \nu_{i_0}$, then there exists $i\geq i_0$ such that $|t|\in[\nu_i,\nu_{i+1}]$. Then $|t|\geq\nu_i\geq N_1(i)$, so that $0\leq f_0(t)\leq e^{-i^2}<1$. This implies that $0\leq f_1(t), f_2(t)\leq f_0(t)$.  We then have
\begin{align}\label{eq:uniform_continuity_two}
\begin{split}
|[x_n(f_1)-x_n(f_2)](t)|&=|a^{-n}(t)[f_1(t)-f_2(t)]|\\
&\leq(1-r)^{-in}|f_1(t)-f_2(t)|\\
&\leq 2(1-r)^{-in} f_0(t)\\
&\leq 2(1-r)^{-in}e^{-i^2}\\
&<\eta.
\end{split}
\end{align}
For the remaining values of $t$, i.e.\ for $t$ such that $|t|\in(\nu_1,\nu_{i_0})$, we have
\begin{align}\label{eq:uniform_continuity_three}
\begin{split}
|[x_n(f_1)-x_n(f_2)](t)|&=|a^{-n}(t)[f_1(t)-f_2(t)]|\\
& \leq\left(\max_{|t|\in[\nu_1,\nu_{i_0}]}a^{-n}(t)\right)\norm{f_1-f_2}.
\end{split}
\end{align}
It follows from \eqref{eq:uniform_continuity_one}, \eqref{eq:uniform_continuity_two}, and \eqref{eq:uniform_continuity_three} that $\norm{x_n(f_1)-x_n(f_2)}<\eta$ for all $f_1,f_2\in S$ such that $\norm{f_1-f_2}<\eta/\max_{|t|\in[\nu_1,\nu_{i_0}]}a^{-n}(t)$.
This establishes our claim concerning the uniform continuity in part (3) of Theorem~\ref{thm:simultaneous_power_factorization}. Since the inverse of $x_n$, i.e.\ the restriction of multiplication with $a^n$, is the restriction of a continuous map,  $x_n:S\to x_n(S)$ is a homeomorphism for $n\geq 1$.

Turning to part (4a) of Theorem~\ref{thm:simultaneous_power_factorization}, let $\epsilon>0$ and $n_0\geq 1$ be given. We note that \eqref{eq:gaussian_limit_is_zero} implies that there exists $K> 0$ such that
\[
0\leq[(1-r)^{-in}-1]e^{-i^2}\leq K
\]
for all $n=1,\ldots,n_0$ and $i\geq 1$. We now select a strictly increasing sequence $\{N_2(i)\}_{i=1}^\infty$ of integers $N_2(i)\geq 1$ such that $0\leq f_0(t)\leq (\epsilon/K)e^{-i^2}$ for all $t$ such that $|t|\geq N_2(i)$. In addition to our first assumption, we shall assume in the remainder of this example that $\{\nu_i\}_{i=1}^\infty$ is such that $0\leq f_0(t)\leq 1$ whenever $|t|\geq \nu_1$, and also such that $\nu_i\geq N_2(i)$ for all $i\geq 1$. We claim that then $\norm{f-x_n(f)}\leq\epsilon$ for all $n=1,\ldots,n_0$ and $f\in S$.  To see this, we fix $f\in S$ and $n$ such that $1\leq n\leq n_0$.  If $t$ is such that $|t|\leq \nu_1$, then $|a^{-n}(t)f(t)-f(t)|=|f(t)-f(t)|=0$. For other values of $t$, there exists $i\geq 1$ such that $|t|\in[\nu_i,\nu_{i+1}]$. Since then $|t|\geq\nu_i\geq \nu_1$, we have $0\leq f_0(t)\leq 1$, so that $0\leq f(t)\leq f_0(t)$. Using this, and also that $0\leq f_0(t)\leq (\epsilon/K)e^{-i^2}$ since $|t|\geq\nu_i\geq N_2(i)$, we see that
\begin{align*}
||f-x_n(f)](t)|&=|a^{-n}(t)f(t)-f(t)|\\
&=a^{-n}(t)f(t)-f(t)\\
& \leq [(1-r)^{-in}-1]f(t)\\
& \leq [(1-r)^{-in}-1]f_0(t)\\
& \leq [(1-r)^{-in}-1]\frac{\epsilon}{K} e^{-i^2}\\
&\leq\epsilon.
\end{align*}
This establishes our claim concerning part (4a) of Theorem~\ref{thm:simultaneous_power_factorization}.

We now turn to part (4b) of Theorem~\ref{thm:simultaneous_power_factorization}. Let $\{\alpha_n\}_{n=1}^\infty\subset(1,\infty)$ such that $\lim_{n\to\infty} \alpha_n=
\infty$ and $\delta>0$ be given; we may assume that $\delta\leq 1$. If $i\geq 1$ is fixed, then, since $\lim_{n\to\infty} \alpha_n=\infty$, we can choose  an integer $N^\prime(i)\geq 1$ such that $0\leq(1-r)^{-i}\leq\alpha_n$ for all $n\geq N^\prime(i)$. We now select a strictly increasing sequence $\{N_3(i)\}_{i=1}^\infty$ of integers $N_3(i)\geq 1$ such $(1-r)^{-in}f_0(t)\leq\alpha_n^n\delta$ for all $n=1,\ldots,N^\prime(i)-1$ and all $t$ such that $|t|\geq N_3(i)$. In addition to our first and second assumption, we shall assume in the remainder of this example that $\{\nu_i\}_{i=1}^\infty$ is such that $0\leq f_0(t)\leq \delta$ for all $t$ such that $|t|\geq \nu_1$, and also such that $\nu_i\geq N_3(i)$ for all $i\geq 1$.
We claim that then $\norm{x_n(f)}\leq\alpha_n^n\max(\norm{f},\delta)$ for all $n\geq 1$ and $f\in S$. To see this, we fix $n\geq 1$ and $f\in S$. If $|t|\leq\nu_1$, then $a(t)=1$, which implies that
\begin{align*}
|[x_n(f)](t)|&=|a^{-n}(t)f(t)|\\
&=f(t)\\
&\leq\norm{f}\\
&\leq\max(\norm{f},\delta)\\
&<\alpha_n^n\max(\norm{f},\delta).
\end{align*}
If $|t|>\nu_1$, then there exists $i\geq 1$ such that $|t|\in [\nu_i,\nu_{i+1}]$. Then $|t|\geq\nu_i\geq\nu_1$, so that $0\leq f_0(t)|\leq\delta\leq 1$, implying that $0\leq f(t)\leq f_0(t)\leq\delta$. We now distinguish between two cases how our fixed $n$ can be related to $N^\prime(i)$ for this particular $i$. If $n=1,\ldots,N^\prime(i)-1$, then
\begin{align*}
|[x_n(f)](t)|&=|a^{-n}(t)f(t)|\\
&\leq (1-r)^{-in} f(t)\\
&\leq (1-r)^{-in}f_0(t)\\
&\leq \alpha_n^n\delta\\
&\leq\alpha_n^n\max(\norm{f},\delta),
\end{align*}
where we have used that $|t|\geq\nu_i\geq N_3(i)$ in the fourth step. If $n\geq N^\prime(i)$, then
\begin{align*}
|[x_n(f)](t)|&=|a^{-n}(t)f(t)|\\
&\leq (1-r)^{-in}f(t)\\
&\leq \alpha_n^n f(t)\\
&\leq\alpha_n^{n}\delta\\
&\leq\alpha_n^n \max(\norm{f},\delta).
\end{align*}
Our claim concerning part (4b) of Theorem~\ref{thm:simultaneous_power_factorization} has now been established.

It is obvious that the `linearity' of the maps $x_n$ for $n\geq 1$ in part (5) of Theorem~\ref{thm:simultaneous_power_factorization} holds.

We shall consider part (6) in a moment, but we treat part (7) of Theorem~\ref{thm:simultaneous_power_factorization} first. The parts (7a) and (7b) are not applicable. Part (7c) is easily verified, once one notes that $S$ is invariant under multiplication by $e_\nu$ for all $\nu\geq 1$. Fix $n\geq 1$. Using our observation in the final step, one can then write, for $f\in S$ and $\nu\geq 1$,
\[
\norm{e_\nu x_n(f)- x_n(f)}=\norm{e_\nu a^{-n}f\!-\!a^{-n}f}=\norm{a^{-n}e_\nu f - a^{-n}f}=\norm{x_n(e_\nu f) -x_n(f)}.
\]
Since  $\lim_{\nu\to\infty}e_\nu f = f$ uniformly on $S$, and since we have also already established that $x_n$ is uniformly continuous on $S$, we now see that $\lim_{\nu\to\infty} e_\nu x_n(f)=x_n(f)$ uniformly on $S$.

Finally, we turn to part (6) of Theorem~\ref{thm:simultaneous_power_factorization}. The parts (6a) and (6c) were built into our construction from the very start, and we are left with part (6b). For $k\geq 1$, let the element $b_k\in\contb{\RR}$ be defined by
\begin{equation}\label{eq:b_k_definition}
b_k=(1-r)^{k}\onefunction + \sum_{i=1}^{k}r(1-r)^{i-1}e_{\nu_i}.
\end{equation}
The $b_k$ agree with $a$ on ever larger intervals. More precisely, an easy pointwise calculation and comparison of the results with \eqref{eq:a_value_1}, \eqref{eq:a_value_transition}, and \eqref{eq:a_value_interval}  yield that, for $k\geq 1$,

\begin{align}
b_k(t)&=a(t)&&\textup{ if }|t|\leq \nu_{k+1},\label{eq:b_k_equals_a}
\intertext{and}
b_k(t)&=(1-r)^k&&\textup{ if }|t|\geq\nu_{k+1}.\label{eq:b_k_remaining_part}
\end{align}
The $b_k$ are clearly invertible in $\contb{\RR}$. We claim that $x_n(f)=\lim_{k\to\infty}b_k^{-n}f$ for all $n\geq 1$ and $f\in S$. To see this, fix $n\geq 1$ and $f\in S$. We see from \eqref{eq:a_value_transition}, \eqref{eq:a_value_interval}, \eqref{eq:b_k_equals_a}, and \eqref{eq:b_k_remaining_part} that $a(t)=b_k(t)$ if $|t|\leq\nu_{k+1}$, and that $0\leq a(t)\leq b_k(t)$ if $|t|\geq\nu_{k+1}$. Therefore,
\begin{align}\label{eq:powers_of_a_and_b_k_pointwise_comparison}
\begin{split}
\norm{x_n(f)-b_k^{-n}f}&=\sup_{t\in\RR}|a^{-n}(t)f(t)-b_k^{-n}(t)f(t)|\\
&=\sup_{|t|\geq\nu_{k+1}}|a^{-n}(t)f(t)-b_k^{-n}(t)f(t)|\\
&\leq 2 \sup_{|t|\geq\nu_{k+1}}a^{-n}(t)f(t).
\end{split}
\end{align}
Since we have already established that $a^{-n}f\in\conto{\RR}$ under our first assumed condition on the sequence $\{\nu_i\}_{i=1}^\infty$, it follows from \eqref{eq:powers_of_a_and_b_k_pointwise_comparison} that $\lim_{k\to\infty}\norm{x_n(f)-b_k^{-n}f}=0$. This establishes our claim on pointwise convergence.

Our final task is now to  establish that every $b_k^{-1}$ is an element of the unital Banach subalgebra of $\contb{\RR}$ that is generated by $\onefunction$ and $e_{\nu_1},\ldots,e_{\nu_k}$. A natural first attempt is to use \eqref{eq:b_k_definition}, and investigate whether
\[
\left\Vert(1-r)^k\onefunction - b_k\right\Vert<\norm{[(1-r)^k\onefunction]^{-1}}^{-1}.
\]
This would imply what we want, but unfortunately this is equivalent to showing that
\[
1-(1-r)^k<(1-r)^k,
\]
which, for every $r$ such that $0<r<1$, is false if $k$ is large enough. So we have to proceed differently.

To this end, we note that, not dissimilar from \eqref{eq:maximal_element_factorization},
\begin{equation}\label{eq:b_k_product}
b_k=\prod_{i=1}^k{(1-r+r e_{\nu_i})}.
\end{equation}
Indeed, once one notices that $e_{\nu_i}e_{\nu_{i^\prime}}=e_{\nu_i}$ if $i^\prime\geq i$, it easily established by induction (also similar to the induction that one uses to establish \eqref{eq:maximal_element_factorization}) that the right hand side of \eqref{eq:b_k_product} equals the right hand side of \eqref{eq:b_k_definition}.

It is the factorization of $b_k$ in \eqref{eq:b_k_product} that enables us to show that every $b_k^{-1}$ is in the Banach subalgebra as described, provided that we assume that $0<r<1/2$. Whereas the assumption that $0<r<1$ was sufficient so far, it is here, at the very end of this example, that we finally want to restrict $r$ to the interval that is specified in Theorem~\ref{thm:simultaneous_power_factorization}. The reason is simply that, as in the proofs leading to Theorem~\ref{thm:simultaneous_power_factorization}, we want to be able to write down a Neumann series for the inverse of each factor $(1-r+r e_{\nu_i})$ in \eqref{eq:b_k_product}. This will clearly imply that $b_k^{-1}$ is in the Banach subalgebra as described. We now merely need to note that this can be done as
\[
(1-r+r e_{\nu_i})^{-1}=\frac{1}{1-r}\sum_{j=0}^\infty \left(\frac{r}{r-1}e_{\nu_i}\right)^j,
\]
provided that $\norm{\frac{r}{r-1}e_{\nu_i}}=\frac{r}{1-r}<1$, i.e. provided that $r<1/2$.

This concludes our example of the explicit construction of a positive simultaneous power factorization that has all properties as in Theorem~\ref{thm:simultaneous_power_factorization}.

\begin{remark}\quad
\begin{enumerate}
\item The restriction of the above factorization to the subset $\{\,f\in\conto{\RR} : 0\leq f(t)\leq f_0(t)\,\}$ of $S$ is an explicit positive simultaneous power factorization for the order interval $[0,f_0]$ in $\conto{\RR}$. A similar construction will give an explicit simultaneous power factorization for all order intervals $[f_1,f_2]$ in $\conto{\RR}$.
\item It seems likely that it is possible to generalize the above example to $\conto{\ts}$ for an arbitrary locally compact Hausdorff space $\ts$, possibly requiring a little extra technique. We leave it to the diligent reader to undertake such an endeavour.
\end{enumerate}
\end{remark}

\subsection*{Acknowledgements} The authors thank Garth Dales and Miek Messerschmidt for helpful discussions. During this research, the second author was supported by a China Scholarship Council grant.

\bibliography{pmybibfile}

\end{document}